\documentclass[reqno]{amsart}

\usepackage{amsmath,amssymb,amsthm,amscd}
\usepackage[mathscr]{eucal}
\usepackage{color}
\usepackage{epic,eepic}
\usepackage{bm}
\usepackage[all]{xy}
\usepackage{comment}

\newtheorem{theorem}{Theorem}
\newtheorem{lemma}[theorem]{Lemma}

\newtheorem{proposition}[theorem]{Proposition}

\theoremstyle{definition}

\renewcommand{\labelenumi}{(\theenumi)}

\newcommand{\C}{\mathbb{C}}

\newcommand{\Z}{\mathbb{Z}}

\newcommand{\calC}{\mathcal{C}}
\newcommand{\calE}{\mathcal{E}}
\newcommand{\calP}{\mathcal{P}}
\newcommand{\calR}{\mathcal{R}}
\newcommand{\calZ}{\mathcal{Z}}

\newcommand{\geh}{\mathfrak{g}}
\newcommand{\heh}{\mathfrak{h}}
\newcommand{\enu}{\mathfrak{n}}
\newcommand{\ep}{\epsilon}

\newcommand{\La}{\Lambda}

\newcommand{\ot}{\otimes}
\newcommand{\op}{\oplus}

\newcommand{\pair}[1]{\langle{#1}\rangle}

\begin{document}

\title[]{Vertex algebraic construction of modules for twisted affine Lie algebras of type $A_{2l}^{(2)}$}

\author{RYO TAKENAKA}
\address{Department of Mathematics, Osaka Metropolitan University, Osaka 558-8585, Japan}
\email{r.takenaka0419@gmail.com}

\begin{abstract}
Let $\tilde{\geh}$ be the affine Lie algebra of type $A_{2l}^{(2)}$.
The integrable highest weight $\tilde{\geh}$-module $L(k\La_0)$ called the standard $\tilde{\geh}$-module is realized by a tensor product of the twisted module $V_L^T$ for the lattice vertex operator algebra $V_L$.
By using such vertex algebraic construction, we construct bases of the standard module, its principal subspace and the parafermionic space.
As a consequence, we obtain their character formulas and settle the conjecture for vacuum modules stated in \cite{HKOTT}.
\end{abstract}

\maketitle

\section{Introduction}

In 1970s-1980s, Lepowsky, Primc and Wilson studied the integrable highest weight modules called standard modules using vertex operators to obtain a number of combinatorial identities.
See e.g. \cite{LP, LW2}.
Lepowsky and Primc showed that the structure of the standard module is completely determined by the structure of its vacuum space which is isomorphic to a coset subspace of the standard module.
Let $\geh$ be the simple Lie algebra $\mathfrak{sl}_2$ and $\heh$ be its Cartan subalgebra.
The corresponding affine Lie algebras are given by
\[\tilde{\geh}=\geh\ot\C[t,t^{-1}]\op\C c\op\C d,\quad\tilde{\heh}=\heh\ot\C[t,t^{-1}]\op\C c\op\C d,\]
where $c$ is the canonical central element and $d$ is the degree operator.
They constructed combinatorial bases of the coset space $\tilde{\geh}\supset\tilde{\heh}$ of standard $\tilde{\geh}$-modules build upon Fourier coefficients of vertex operators.
This work was generalized by Georgiev to the higher rank case $\geh=\mathfrak{sl}_{n+1}$.
In \cite{G}, he first constructed bases for the principal subspaces introduced and studied by Feigin and Stoyanovsky \cite{FS,FS2}.
Then he did for certain subspaces of the standard modules called parafermionic spaces by generalizing the $\calZ$-algebra approach of Lepowsky, Primc and Wilson.
As a consequence, he obtained a fermionic character formula of the standard module $L(k_0\La_0+k_j\La_j)$.

Let $\tilde{\geh}$ be an affine Lie algebra of type $X_l^{(r)}$ with Weyl group $W$.
Denote $\Delta_+,{\rm mult}(\alpha)\text{ and }\ell(w)$ by the set of positive roots, the multiplicity of $\alpha\in\Delta_+$ and the length of $w\in W$ respectively.
It is well known that the character of the standard $\tilde{\geh}$-module $L(\La)$ is given by
\[{\rm ch}\ L(\La)=\frac{\sum_{w\in W}(-1)^{\ell(w)}e(w(\La+\rho)-\rho)}{\prod_{\alpha\in\Delta_+}(1-e(-\alpha))^{{\rm mult}(\alpha)}},\]
where $e(\cdot)$ is the formal exponential and $\rho$ is an element of $\heh^*$ such that
\[\pair{\rho,\alpha_i^\vee}=1\quad(i=1,\ldots,l)\]
for simple coroots $\alpha_i^\vee(i=1,\ldots,l)$.
See \cite{K}.
In contrast to this Weyl-Kac character formula, there exists a so called fermionic character formula which is written in terms of $q$ series without negative coefficients, where $q=e(-\delta)$ and $\delta$ is a generator of imaginary roots.
In the particular case of the vacuum module, conjectures of fermionic character formulas are stated in \cite{HKOTT,KNS}.

Recently, based on the study of the construction of principal subspaces for untwisted affine Lie algebras (see e.g. \cite{B,BK,G}), Butorac, Ko\v{z}i\'{c} and Primc obtained combinatorial bases of standard modules for all untwisted affine Lie algebras in \cite{BKP}.
Furthermore, by considering parafermionic spaces, they settled the Kuniba-Nakanishi-Suzuki conjecture \cite{KNS}.
On the other hand, Okado and the author constructed bases of standard modules for twisted affine Lie algebras except for type $A_{2l}^{(2)}$ by using the results for principal subspaces obtained by Butorac and Sadowski \cite{BS}.
They were thereby able to show partly fermionic character formulas for twisted cases conjectured in \cite{HKOTT}.

The aim of this paper is to calculate the fermionic character formula of the standard moduel for the twisted affine Lie algebra of type $A_{2l}^{(2)}$, that is, to complete the proof of \cite[conjecture 5.3]{HKOTT} for all twisted affine Lie algebras.
In order to do that, we first need to obtain the quasi-particle basis of the principal subspace of the standard module by generalizing the seminal works of Calinescu, Lepowsky, Milas and Penn \cite{CLM,CMP} to higher rank and level cases.
In the same way as \cite{OT}, we are able to construct the combinatorial basis of the standard module.
In addition to these, we construct the parafermionic basis of the parafermionic space.
Finally, we calculate fermionic character formulas for the principal subspace and the parafermionic space as well as the standard module.

\section{Preliminaries}

\subsection{Lattice vertex operator algebras and twisted modules}
Let $\geh$ be a complex simple Lie algebra of type $A_{2l}$, and $\alpha_i\ (i=1,2,\ldots,2l)$ be its simple roots.
We recall the vertex algebraic construction associated to the root lattice of $\geh$ given by
\[L=\Z\alpha_1\op\cdots\op\Z\alpha_{2l}.\]
Its Dynkin diagram and the Dynkin automorphism are given in Table \ref{Dynkin auto}.
\begin{table}[h]
\begin{align*}
\xymatrix{*{\overset{1}{\bigcirc}}\ar@{<->}[d]\ar@<-0.8ex>@{-}[r]&*{\overset{2}{\bigcirc}}\ar@{<->}[d]\ar@<-0.8ex>@{--}[r]&*{\overset{l-1}{\bigcirc}}\ar@{<->}[d]\ar@<-0.8ex>@{-}[r]&*{\overset{l}{\bigcirc}}\ar@{-}[d]\ar@/^18pt/[d]\\
*{\underset{2l}{\bigcirc}}\ar@{-}@<0.8ex>[r]&*{\underset{2l-1}{\bigcirc}}\ar@<0.8ex>@{--}[r]&*{\underset{l+2}{\bigcirc}}\ar@<0.8ex>@{-}[r]&*{\underset{l+1}{\bigcirc}}\ar@/_18pt/[u]}
\end{align*}
\caption{Dynkin diagram of $\geh$ and the automorphism}\label{Dynkin auto}
\end{table}

\noindent Thus the automorphism $\nu$ of $L$ is determined by
\[\nu(\alpha_i)=\alpha_{2l-i+1}.\]
Let $\pair{\cdot,\cdot}$ be a nondegenerate invariant symmetric bilinear form on $\geh$.
Using this form, we identify the Cartan subalgebra $\heh$ of $\geh$ with its dual $\heh^*$, so that under this identification we have $L\subset\heh$.
We fix the bilinear form $\pair{\cdot,\cdot}$ so that we have $\pair{\alpha,\alpha}=2$ if $\alpha$ is a root.
Note that
\[\pair{\nu^\frac{r}{2}\alpha,\alpha}\in2\Z\quad\text{for all }\alpha\in L,\]
under the assumptions in Section 2 of \cite{L}.
Therefore we choose $r$ to be 4 rather than 2.
In fact, the period of $\nu$ is allowed to be larger than the order of $\nu$ in our setting (see the remark in Section 2 of \cite{L}).
Let $\zeta$ be the imaginary unit.

Following \cite{CLM, CMP, L}, we define the functions $C_0$, $C:L\times L\rightarrow\C^\times$ by
\[C_0(\alpha,\beta)=(-1)^{\pair{\alpha,\beta}},\quad C(\alpha,\beta)=\prod_{j=0}^3(-\zeta^j)^{\pair{\nu^j\alpha,\beta}}.\]
These functions are bilinear into the abelian group $\C^\times$ and $\nu$-invariant.
Since $C_0$ and $C$ satisfy
\[C_0(\alpha,\alpha)=C(\alpha,\alpha)=1\]
for all $\alpha\in L$, they determine uniquely two central extensions of $L$ by $\pair{\zeta}$ with commutator maps $C_0$ or $C$ denoted by $\hat{L}$ or $\hat{L}_\nu$,
\[1\longrightarrow\pair{\zeta}\longrightarrow\hat{L}~(\text{ or }\hat{L}_\nu)~\bar{\longrightarrow} L\longrightarrow1\]
where $\bar{\phantom{a}}$ stands for the projection to $L$.
Namely, we have $aba^{-1}b^{-1}=C_0(\overline{a},\overline{b})$ (resp. $C(\overline{a},\overline{b})$) for $a,b\in\hat{L}$ (resp. $\hat{L}_\nu$).
Let each commutator map correspond to 2-cocycle $\ep_{C_0}$ and $\ep_C$.
That is $\ep_{C_0}$ satisfies
\[\ep_{C_0}(\alpha,\beta)\ep_{C_0}(\alpha+\beta,\gamma)=\ep_{C_0}(\beta,\gamma)\ep_{C_0}(\alpha,\beta+\gamma),\quad \frac{\ep_{C_0}(\alpha,\beta)}{\ep_{C_0}(\beta,\alpha)}=C_0(\alpha,\beta).\]
We choose our 2-cocycle $\ep_{C_0}$ to be
\[\ep_{C_0}(\alpha_i,\alpha_j)=
\begin{cases}
1&\text{if }i\leq j\\
(-1)^{\pair{\alpha_i,\alpha_j}}&\text{if }i>j
\end{cases}.\]
This 2-cocycle satisfies
\[\ep_{C_0}(\alpha,\beta)^2=1,\quad\ep_{C_0}(\alpha,\beta)=\ep_{C_0}(\nu\alpha,\nu\beta).\]
Also, the 2-cocycle $\ep_C$ is given by
\begin{align}\label{cocycle}
\ep_{C_0}(\alpha,\beta)=(-\zeta)^{\pair{\nu^{-1}\alpha,\beta}}\ep_C(\alpha,\beta).
\end{align}
Using these 2-cocycles, we obtain two normalized sections $e:L\rightarrow\hat{L}~(\text{resp. }\hat{L}_\nu)$ by
\begin{eqnarray*}
e:L&\rightarrow&\hat{L}~(\text{resp. }\hat{L}_\nu)\\
\alpha&\mapsto&e_\alpha
\end{eqnarray*}
with $e_0=1$, $\overline{e_\alpha}=\alpha$ and $e_\alpha e_\beta=\ep_{C_0}(\alpha,\beta)e_{\alpha+\beta}$ $(\text{resp. }\ep_C(\alpha,\beta)e_{\alpha+\beta})$.

According to \cite{CLM},  there exists an automorphism $\hat{\nu}$ of $\hat{L}$ such that
\begin{align}\label{ass}
\overline{\hat{\nu}a}=\nu\overline{a},\quad\hat{\nu}a=a~\text{if }\nu\overline{a}=\overline{a}.
\end{align}
In order to write down this automorphism explicitly, we will use the following notation defined in \cite{CMP}.
Set
\[\alpha_i^{(j)}=\alpha_i+\cdots+\alpha_{i+j-1}=\sum_{k=0}^{j-1}\alpha_{i+k}\]
for the simple roots of $\geh$.
We have the set of roots of $\geh$ as follows.
\begin{align}\label{root}
\Delta=\{\pm\alpha_i^{(j)}\mid 1\leq i\leq 2l,1\leq j\leq 2l-i+1\}.
\end{align}
For this notation, we have
\[\nu(\alpha_i^{(j)})=\alpha_{2l-i-j+2}^{(j)}.\]
Thus $\alpha_i^{(2l-2i+2)}\ (i=1,\ldots,l)$ are invariant under the automorphism $\nu$.
From \cite{CLM, CMP}, we can choose $\hat{\nu}$ to be
\begin{align*}
\hat{\nu}(e_{\pm\alpha_i})=&-e_{\pm\alpha_{2l-i+1}}\text{ if }i\notin\{l,l+1\},\\
\hat{\nu}(e_{\pm\alpha_l})=&\pm\zeta e_{\pm\alpha_{l+1}},\\
\hat{\nu}(e_{\pm\alpha_{l+1}})=&\pm\zeta e_{\pm\alpha_l}.
\end{align*}
As in \cite{CMP}, we say that $\alpha_i^{(j)}$ contains $\alpha_m$ if $i\leq m\leq i+j-1$.
Now we have the following proposition (See \cite{CMP}).
\begin{proposition}\label{CMP prop}
The automorphism $\hat{\nu}$ of $\hat{L}$ is completely determined as follows.
\begin{align}
\hat{\nu}\left(e_{\pm\alpha_i^{(j)}}\right)&=-e_{\pm\alpha_{2l-i-j+2}^{(j)}}\hspace{0.55cm}\text{if }\alpha_i^{(j)}\text{ does not contain }\alpha_l\text{ or }\alpha_{l+1},\label{Lauto}\\
\hat{\nu}\left(e_{\pm\alpha_i^{(j)}}\right)&=\pm\zeta e_{\pm\alpha_{2l-i-j+2}^{(j)}}\quad\text{if }\alpha_i^{(j)}\text{ contains exactly one of }\alpha_l\text{ or }\alpha_{l+1},\label{Lauto2}\\
\hat{\nu}\left(e_{\pm\alpha_i^{(j)}}\right)&=e_{\pm\alpha_{2l-i-j+2}^{(j)}}\hspace{0.7cm}\text{ if }\alpha_i^{(j)}\text{ contains both }\alpha_l\text{ and }\alpha_{l+1}.\label{Lauto3}
\end{align}
\end{proposition}
From Proposition \ref{CMP prop}, we have $\hat{\nu}\left(e_{\alpha_i^{(2l-2i+2)}}\right)=e_{\alpha_i^{(2l-2i+2)}}$ for $i=1,\ldots,l$.
In other words, our automorphism $\hat{\nu}$ satisfies (\ref{ass}) as desired.
The map $\hat{\nu}$ is also an automorphism of $\hat{L}_\nu$ satisfying (\ref{ass}).
We have $\hat{\nu}^4=1$.

We have the affine Lie algebra corresponding to $\heh$ by
\[\hat{\heh}=\heh\ot\C[t,t^{-1}]\oplus\C c\]
with Lie bracket given by
\[[\alpha\ot t^m,\beta\ot t^n]=\pair{\alpha,\beta}m\delta_{m+n,0}c,\quad[\hat{\heh},c]=0\]
for $m,n\in\Z$, $\alpha,\beta\in\heh$.
This affine Lie algebra has the $\Z$-gradation called the weight grading given by
\[{\rm wt}(\alpha\ot t^m)=-m,\quad{\rm wt}(c)=0\]
for $m\in\Z$ and $\alpha\in\heh$.
Consider the subalgebras $\hat{\heh}^\pm=\heh\ot t^{\pm1}\C[t^{\pm1}]$.
Then the Heisenberg subalgebra of $\hat{\heh}$ is given by $\hat{\heh}_\Z=\hat{\heh}^+\oplus\hat{\heh}^-\oplus\C c$.
We introduce the induced $\hat{\heh}$-module
\begin{align}\label{ind mod}
M(1)=U(\hat{\heh})\ot_{U(\heh\ot\C[t]\oplus\C c)}\C\simeq {\rm Sym}(\hat{\heh}^-),
\end{align}
where $\heh\ot\C[t]$ acts trivially on $\C$ and $c$ acts as 1.
$M(1)$ is an irreducible $\hat{\heh}_\Z$-module and $\Z$-graded.
Then we consider the induced $\hat{L}$-module
\[\C\{L\}=\C[\hat{L}]\ot_{\C[\pair{\zeta}]}\C\simeq\C[L].\]
For $a\in\hat{L}$, we write $\iota(a)=a\ot1\in\C\{L\}$.
The space $\C\{L\}$ is $\Z$-graded by letting weight be
\[{\rm wt}\ \iota(a)=\frac{1}{2}\pair{\overline{a},\overline{a}},\quad{\rm wt}(1)=0\]
for $a\in\hat{L}$.
The action of $\hat{L},\heh$ on $\C\{L\}$ is given by
\[a\cdot\iota(b)=\iota(ab),\quad h\cdot\iota(a)=\pair{h,\overline{a}}\iota(a)\]
for $a,b\in\hat{L}$, $h\in\heh$.
We also define the operator $z^h$ on $\C\{L\}$ by
\[z^h\iota(a)=z^{\pair{h,\overline{a}}}\iota(a).\]
We set
\[V_L=M(1)\ot\C\{L\}\simeq{\rm Sym}(\hat{\heh}^-)\ot\C[L]\]
and ${\bf 1}=1\ot\iota(1)$.
Note that $V_L$ has the tensor product grading and is a tensor product of $\hat{\heh}$-module on which $\hat{L}$ acts by its action on the second component.

We consider the vertex operator acting on $V_L$.
For $\alpha\in\heh$, $m\in\Z$, we set $\alpha(m)=\alpha\ot t^m$ and
\[\alpha(z)=\sum_{m\in\Z}\alpha(m)z^{-m-1}.\]
As in \cite{FLM}, we define
\[Y(\iota(a),z)=\overset{\circ}{\underset{\circ}{\phantom{a}}}e^{\sum_{m\neq0}\frac{-\overline{a}(m)}{m}z^{-m}}\overset{\circ}{\underset{\circ}{\phantom{a}}}az^{\overline{a}}\]
for $a\in\hat{L}$, where $\overset{\circ}{\underset{\circ}{\phantom{a}}}\cdot\overset{\circ}{\underset{\circ}{\phantom{a}}}$ means the normally ordered product in which the operators $\overline{a}(m)$ for $m<0$ are placed to the left of the operators $\overline{a}(m)$ for $m>0$.
More generally, for a vector $v=\beta_1(-n_1)\cdots\beta_m(-n_m)\ot\iota(a)\in V_L$ with $\beta_1,\ldots,\beta_m\in\heh$, $n_1,\ldots,n_m>0$ and $a\in\hat{L}$ we set
\[Y(v,z)=\overset{\circ}{\underset{\circ}{\phantom{a}}}\left(\frac{1}{(n_1-1)!}\left(\frac{d}{dz}\right)^{n_1-1}\beta_1(z)\right)\cdots\left(\frac{1}{(n_m-1)!}\left(\frac{d}{dz}\right)^{n_m-1}\beta_m(z)\right)Y(\iota(a),z)\overset{\circ}{\underset{\circ}{\phantom{a}}}.\]
This gives a well-defined linear map
\begin{eqnarray*}
V_L&\rightarrow&({\rm End}\ V_L)[[z,z^{-1}]]\\
v&\mapsto&Y(v,z)=\sum_{m\in\Z}v_mz^{-m-1},\quad v_m\in{\rm End}\ V_L.
\end{eqnarray*}
Set
\[\omega=\frac{1}{2}\sum_{i=1}^{2l}\gamma_i(-1)\gamma_i(-1){\bf 1}\in V_L,\]
where $\{\gamma_i\}_{i=1}^{2l}$ is an orthonormal basis of $\heh$.

By Chapter 8 in \cite{FLM2}, $(V_L,Y,{\bf 1},\omega)$ becomes a simple vertex operator algebra associated to the positive definite even lattice $L$ equipped with central charge equal to ${\rm rank}\ L=2l$.

Now we extend the automorphism $\hat{\nu}$ of $\hat{L}$ to an automorphism of $V_L$ and also denote it by $\hat{\nu}$ (cf. \cite{CLM}).
Note that the automorphism $\nu$ of $L$ acts in a natural way on $\hat{\heh}$ and $M(1)$, preserving the grading.
We have
\[\nu(u\cdot m)=\nu(u)\cdot\nu(m)\]
for $u\in\hat{\heh}$ and $m\in M(1)$.
Also the automorphism $\hat{\nu}$ of $\hat{L}$ is extended to $\C\{L\}$ satisfying conditions
\[\hat{\nu}(h\cdot\iota(a))=\nu(h)\cdot\hat{\nu}(\iota(a)),\ \hat{\nu}(z^h\cdot\iota(a))=z^{\nu(h)}\cdot\hat{\nu}(\iota(a))\text{ and }\hat{\nu}(a\cdot\iota(b))=\hat{\nu}(a)\cdot\hat{\nu}(\iota(b))\]
for $h\in\heh$ and $a,b\in\hat{L}$.
We take $\hat{\nu}$ on $V_L$ to be $\nu\ot\hat{\nu}$.
It follows that $\hat{\nu}$ is an automorphism of the vertex operator algebra $V_L$ which preserves the grading on $V_L$.

We construct the $\hat{\nu}$-twisted module for $V_L$ by following \cite{CLM,DL, L}.
For $j\in\Z$, set
\[\heh_{(j)}=\{h\in\heh\mid\nu h=\zeta^jh\}\subset\heh,\]
so that we have
\[\heh=\bigoplus_{j\in\Z/4\Z}\heh_{(j)}.\]
Here we identify $\heh_{(j\text{ mod }4)}$ with $\heh_{(j)}$.
Associated to this decomposition, we define a Lie algebra
\[\hat{\heh}[\nu]=\bigoplus_{m\in\frac{1}{4}\Z}\heh_{(4m)}\ot t^m\oplus\C c\]
with Lie bracket given by
\[[\alpha\ot t^m,\beta\ot t^n]=\pair{\alpha,\beta}m\delta_{m+n,0}c,\quad [\hat{\heh}[\nu],c]=0\]
for $m,n\in\frac{1}{4}\Z$ and $\alpha\in\heh_{(4m)}$ and $\beta\in\heh_{(4n)}$.
As in the untwisted case, we give the $\frac{1}{4}\Z$-gradation to $\hat{\heh}[\nu]$ by
\[{\rm wt}(\alpha\ot t^m)=-m,\quad {\rm wt}(c)=0.\]
Consider the subalgebras $\hat{\heh}[\nu]^\pm=\bigoplus_{\pm m>0}\heh_{(4m)}\ot t^m$.
Then the Heisenberg subalgebra of $\hat{\heh}[\nu]$ is obtained by $\hat{\heh}[\nu]_{\frac{1}{4}\Z}=\hat{\heh}[\nu]^+\oplus\hat{\heh}[\nu]^-\oplus\C c$.
The induced $\hat{\heh}[\nu]$-module is obtained in similar to (\ref{ind mod}), that is we have
\[S[\nu]=U(\hat{\heh}[\nu])\ot_{U(\bigoplus_{m\geq0}\heh_{(4m)}\ot t^m\oplus\C c)}\C\simeq{\rm Sym}(\hat{\heh}[\nu]^-)\]
where $\bigoplus_{m\geq0}\heh_{(4m)}\ot t^m$ acts trivially on $\C$ and $c$ acts as 1.
This module is also an irreducible $\hat{\heh}[\nu]_{\frac{1}{4}\Z}$-module.

We continue to follow \cite{CLM, L}.
Let $P_j$ be the projection from $\heh$ onto $\heh_{(j)}$ for $j\in\Z/4\Z$.
We set
\[N=(1-P_0)\heh\cap L=\{\alpha\in L\mid\pair{\alpha,\heh_{(0)}}=0\}.\]
Explicitly, we have
\[N={\rm Span}_{\Z}\{\alpha_i-\alpha_{2l-i+1}\mid 1\leq i\leq l\}.\]
See \cite{CMP}.
Let $\hat{N}$ be the subgroup of $\hat{L}_\nu$ obtained by pulling buck the subgroup $N$ of $L$.
By Proposition 6. 1 in \cite{L}, there exists a unique homomorphism $\tau:\hat{N}\rightarrow\C^\times$such that
\[\tau(\zeta)=\zeta,\quad\tau(a\hat{\nu}a^{-1})=\zeta^{-\sum_{j=0}^3\pair{\nu^j\overline{a},\overline{a}}/2}\]
for $a\in\hat{L}_\nu$.
Let $\C_\tau$ be the one-dimensional $\hat{N}$-module $\C$ with this character $\tau$ and consider the induced $\hat{L}_\nu$-module
\[U_T=\C[\hat{L}_\nu]\ot_{\C[\hat{N}]}\C_\tau\simeq\C[L/N].\]
$\hat{L}_\nu$ and $\heh_{(0)}$ act on $U_T$ as
\begin{align}\label{actut1}
a\cdot b\ot t=ab\ot t,\\
\label{actut2}
\alpha\cdot a\ot t=\pair{\alpha,\overline{a}}a\ot t
\end{align}
for $a,b\in\hat{L}_\nu$, $t\in\C_\tau$, $\alpha\in\heh_{(0)}$ and we set $[\alpha,a]=\pair{\alpha,\overline{a}}a$.
The operator $z^h$ on $U_T$ is also defined by
\begin{align}\label{actut3}
z^h\cdot a\ot t=z^{\pair{h,\overline{a}}}a\ot t
\end{align}
for $h\in\heh_{(0)}$.
For $h\in\heh_{(0)}$ such that $\pair{h,L}\subset\Z$, we define the operator $\zeta^h$ on $U_T$ by $\zeta^h\cdot a\ot t=\zeta^{\pair{h,\overline{a}}}a\ot t$.
Then for $a\in\hat{L}_\nu$ we have $z^ha=az^{h+\pair{h,\overline{a}}}$ and $\zeta^ha=a\zeta^{h+\pair{h,\overline{a}}}$.
Moreover, as an operator on $U_T$, we have
\begin{align}\label{autorel}
\hat{\nu}a=a\zeta^{-\sum_{j=0}^3\nu^j\overline{a}-\sum_{j=0}^3\pair{\nu^j\overline{a},\overline{a}}/2}.
\end{align}
Then $U_T$ is decomposed into
\[U_T=\bigoplus_{\alpha\in P_0L}U_\alpha,\]
where $U_\alpha=\{u\in U_T\mid h\cdot u=\pair{h,\alpha}u\text{ for }h\in\heh_{(0)}\}$ satisfying $a\cdot U_\alpha\subset U_{\alpha+\overline{a}_{(0)}}$ for $a\in\hat{L}_\nu$.
Define the $\frac{1}{4}\Z$-grading on $U_T$ by
\begin{align}\label{grad}
{\rm wt}(u)=\frac{1}{2}\pair{\alpha,\alpha}
\end{align}
calculated by $d$ for $u\in U_\alpha$ and $\alpha\in P_0L$.
Set $\tilde{\heh}[\nu]=\hat{\heh}[\nu]\oplus\C d$.
Then $U_T$ becomes an $\tilde{\heh}[\nu]$-module by letting $\hat{\heh}[\nu]_{\frac{1}{4}\Z}\subset\tilde{\heh}[\nu]$ act trivially.
We set
\[V_L^T=S[\nu]\ot U_T\simeq{\rm Sym}(\hat{\heh}[\nu]^-)\ot\C[L/N]\]
to be a tensor product of $\tilde{\heh}[\nu]$-module and ${\bf 1}_T=1\ot(e_0\ot1)$.
Note that $\hat{L}_\nu$ acts by its action on the second component and $V_L^T$ is graded by weights described above.

Next we consider the $\hat{\nu}$-twisted vertex operator by following {\S}2 of \cite{CLM}.
For $\alpha\in\heh$ and $j\in\Z/4\Z$, let $\alpha_{(j)}$ stand for $P_j\alpha\in\heh$.
Set 
\begin{align}
\label{alpha}\alpha^{\hat{\nu}}(m)=\alpha_{(4m)}\ot t^m,\\
\label{TVO}\alpha^{\hat{\nu}}(z)=\sum_{m\in\frac{1}{4}\Z}\alpha^{\hat{\nu}}(m)z^{-m-1}
\end{align}
and
\begin{align}\label{E}
E^\pm(\alpha,z)=\exp\left(\sum_{\pm m\in\frac{1}{4}\Z_+}\frac{\alpha^{\hat{\nu}}(m)}{m}z^{-m}\right)\in{\rm End}\ V_L^T[[z^{\frac{1}{4}},z^{-\frac{1}{4}}]].
\end{align}
Note that from Proposition 3.4 of \cite{LW}, we have the following commutation relation
\begin{align}\label{Ecom}
E^+(\alpha,z)E^-(\beta,w)=E^-(\beta,w)E^+(\alpha,z)\prod_{j=0}^3\left(1-\zeta^j\frac{w^\frac{1}{4}}{z^\frac{1}{4}}\right)^{\pair{\nu^j\alpha,\beta}}.
\end{align}
From the commutation relation of $\hat{\heh}[\nu]$, we also have
\begin{align}\label{Eh}
[E^-(\alpha,z)E^+(\alpha,z),h^{\hat{\nu}}(m)]=\pair{h_{(4m)},\alpha_{(-4m)}}z^mE^-(\alpha,z)E^+(\alpha,z)c
\end{align}
for $h\in\heh$, $m\in\frac{1}{4}\Z$.
For $a\in\hat{L}$, as defined in \cite{CLM, L}, we consider the $\hat{\nu}$-twisted vertex operator
\begin{align}\label{TVO2}
Y^{\hat{\nu}}(\iota(a),z)=2^{-\pair{\overline{a},\overline{a}}}\sigma(\overline{a})E^-(-\overline{a},z)E^+(-\overline{a},z)az^{\overline{a}_{(0)}+\pair{\overline{a}_{(0)},\overline{a}_{(0)}}/2-\pair{\overline{a},\overline{a}}/2}
\end{align}
acting on $V_L^T$, where
\[\sigma(\overline{a})=2^\frac{\pair{\overline{a},\overline{a}}}{2}(1+\zeta)^{\pair{\nu\overline{a},\overline{a}}}.\]
We view $a$ in the right-hand side of (\ref{TVO2}) as an element of $\hat{L}_\nu$ by using the identification between $\hat{L}\text{ and }\hat{L}_\nu$ given by (\ref{cocycle}).
Define the component operators $Y_\alpha^{\hat{\nu}}(m)$ for $\alpha\in L,m\in\frac{1}{4}\Z$ by
\[Y^{\hat{\nu}}(\iota(e_\alpha),z)=\sum_{m\in\frac{1}{4}\Z}Y_\alpha^{\hat{\nu}}(m)z^{-m-\pair{\alpha,\alpha}/2}.\]
For $v=\beta_1(-n_1)\cdots\beta_m(-n_m)\ot\iota(a)\in V_L$, we set
\[W(v,z)=\overset{\circ}{\underset{\circ}{\phantom{a}}}\left(\frac{1}{(n_1-1)!}\left(\frac{d}{dz}\right)^{n_1-1}\beta_1^{\hat{\nu}}(z)\right)\cdots\left(\frac{1}{(n_m-1)!}\left(\frac{d}{dz}\right)^{n_m-1}\beta_m^{\hat{\nu}}(z)\right)Y^{\hat{\nu}}(\iota(a),z)\overset{\circ}{\underset{\circ}{\phantom{a}}}.\]
The map $W:V_L\rightarrow{\rm End}\ V_L^T[[z^\frac{1}{4},z^{-\frac{1}{4}}]]$ gives well-defined linear operator on $V_L^T$ as in the untwisted case.
Recall the operator $\Delta_z$ on $V_L$ defined in \cite{CLM} which is obtained from an orthonormal basis $\{\gamma_i\}_{i=1}^{2l}$ of $\heh$.
Set
\[\Delta_z=\sum_{m,n\geq0}\sum_{j=0}^3\sum_{i=1}^{2l}c_{mnj}(\nu^{-j}\gamma_i)(m)\gamma_i(n)z^{-m-n},\]
where constants $c_{mnj}$ are given by
\[\sum_{m,n\geq0}c_{mn0}z^mw^n=-\frac{1}{2}\sum_{j=1}^3\log\left(\frac{(1+z)^\frac{1}{4}-\zeta^{-j}(1+w)^\frac{1}{4}}{1-\zeta^{-j}}\right),\]
\[\sum_{m,n\geq0}c_{mnj}z^mw^n=\frac{1}{2}\log\left(\frac{(1+z)^\frac{1}{4}-\zeta^j(1+w)^\frac{1}{4}}{1-\zeta^j}\right)\quad\text{for }j\neq0.\]
Note that $\Delta_z$ is independent of the choice of the orthonormal basis.
Since $c_{00j}=0$ for all $j$, the map $e^{\Delta_z}$ becomes well-defined operator on $V_L$ and we have $e^{\Delta_z}v\in V_L[z^{-1}]$ for all $v\in V_L$.
Then, for $v\in V_L$, we have the $\hat{\nu}$-twisted vertex operator as
\[Y^{\hat{\nu}}(v,z)=W(e^{\Delta_z}v,z).\]

By \cite{DL, FLM, FLM2, L}, $(V_L^T,Y^{\hat{\nu}})$ has the structure of an irreducible $\hat{\nu}$-twisted $V_L$-module.
We have the twisted Jacobi identity
\begin{align}\label{Jacobi}
\nonumber z_0^{-1}\delta\left(\frac{z_1-z_2}{z_0}\right)Y^{\hat{\nu}}(u,z_1)Y^{\hat{\nu}}(v,z_2)-z_0^{-1}\delta\left(\frac{z_2-z_1}{-z_0}\right)Y^{\hat{\nu}}(v,z_2)Y^{\hat{\nu}}(u,z_1)\\
=z_2^{-1}\frac{1}{4}\sum_{j=0}^3\delta\left(\zeta^j\frac{(z_1-z_0)^\frac{1}{4}}{z_2^\frac{1}{4}}\right)Y^{\hat{\nu}}(Y(\hat{\nu}^ju,z_0)v,z_2)
\end{align}
for $u,v\in V_L$.
The commutator formula of twisted vertex operator (cf. \cite{FLM}) is derived from this identity.
By using (\ref{autorel}), we have
\begin{align}\label{TVO rel}
Y^{\hat{\nu}}(\hat{\nu}^jv,z)=Y^{\hat{\nu}}(v,z)|_{z^\frac{1}{4}\rightarrow\zeta^{-j}z^\frac{1}{4}}.
\end{align}
See \cite{CLM, L}.
We define the operators $L^{\hat{\nu}}(m)$ for $m\in\Z$ by
\[Y^{\hat{\nu}}(\omega,z)=\sum_{m\in\Z}L^{\hat{\nu}}(m)z^{-m-2}.\]
These operators have the commutation relation
\[[L^{\hat{\nu}}(m),L^{\hat{\nu}}(n)]=(m-n)L^{\hat{\nu}}(m+n)+\frac{l}{6}(m^3-m)\delta_{m+n,0}\]
for $m,n\in\Z$.
That is $\{L^{\hat{\nu}}(m)\mid m\in\Z\}$ generates the Virasoro vertex algebra submodule.
By Proposition 6.3 of \cite{DL}, we have
\[[Y^{\hat{\nu}}(\omega,z_1),Y^{\hat{\nu}}(\iota(a),z_2)]=z_2^{-1}\delta\left(\frac{z_1}{z_2}\right)\frac{d}{dz_2}Y^{\hat{\nu}}(\iota(a),z_2)-\frac{\pair{\alpha,\alpha}}{2}z_2^{-1}\left(\frac{d}{dz_1}\delta\left(\frac{z_1}{z_2}\right)\right)Y^{\hat{\nu}}(\iota(a),z_2)\]
for $a\in\hat{L}$.
We recall from \cite{CLM, DL} that
\[L^{\hat{\nu}}(0)1=\frac{l}{8}1,\quad L^{\hat{\nu}}(0)u=\left(\frac{\pair{\alpha,\alpha}}{2}+\frac{l}{8}\right)u\]
for $1\in S[\nu]$, $u\in U_\alpha\subset V_L^T$ and
\[[L^{\hat{\nu}}(0),\alpha^{\hat{\nu}}(m)]=-m\alpha^{\hat{\nu}}(m)\]
for $\alpha\in\heh_{(4m)}$, $m\in\frac{1}{4}\Z$.
Therefore by using the weight gradings on $S[\nu]$ and $U_T$, we have
\[L^{\hat{\nu}}(0)v=\left({\rm wt}(v)+\frac{l}{8}\right)v\]
for a homogeneous vector $v\in V_L^T$.

\subsection{Representation of twisted affine Lie algebras of type $A_{2l}^{(2)}$ on $V_L^T$}
The aim of this section is to recall the twisted vertex operator construction of affine Lie algebras $A_{2l}^{(2)}$ and give their representation on $V_L^T$.
We have the root space decomposition
\[\geh=\heh\op\bigoplus_{\alpha\in\Delta}\C x_\alpha,\]
where $x_\alpha$ is a root vector such that $[h,x_\alpha]=\alpha(h)$ for $h\in\heh$.
Now, we normalize a root vector $x_\alpha$ so that we have
\[[x_\alpha,x_\beta]=
\begin{cases}
\ep_{C_0}(\alpha,-\alpha)\alpha&\text{if }\alpha+\beta=0\\
\ep_{C_0}(\alpha,\beta)x_{\alpha+\beta}&\text{if }\alpha+\beta\in\Delta\\
0&\text{otherwise.}
\end{cases}\]
Then the symmetric bilinear form $\pair{\cdot,\cdot}$ on $\geh$ reads as
\[\pair{h,x_\alpha}=0,\quad\pair{x_\alpha,x_\beta}=
\begin{cases}\ep_{C_0}(\alpha,-\alpha)&\text{if }\alpha+\beta=0\\
0&\text{if }\alpha+\beta\neq0
\end{cases}\]
As in \cite{CLM, CMP, L}, we introduce the map $\varphi:(\Z/4\Z)\times L\rightarrow\pair{\zeta}$ which is defined by the condition
\[\hat{\nu}^j\iota(e_\alpha)=\varphi(j,\alpha)\iota(e_{\nu^j\alpha}).\]
From the calculation of Proposition \ref{CMP prop}, we have
\begin{align*}
&\varphi(j,\pm\alpha_i^{(j)})=(-1)^j\quad\text{if }\alpha_i^{(j)}\text{ does not contain }\alpha_l\text{ or }\alpha_{l+1},\\
&\varphi(j,\pm\alpha_i^{(j)})=(\pm\zeta)^j\hspace{0.35cm}\text{if }\alpha_i^{(j)}\text{ contains exactly one of }\alpha_l\text{ or }\alpha_{l+1},\\
&\varphi(j,\pm\alpha_i^{(j)})=1\hspace{1.05cm}\text{if }\alpha_i^{(j)}\text{ contains both }\alpha_l\text{ and }\alpha_{l+1}
\end{align*}
for $j\in\Z/4\Z$.
The automorphism $\nu$ of $\heh$ is lifted to an automorphism of $\geh$ by
\[\nu^jx_\alpha=\varphi(j,\alpha)x_{\nu^j\alpha}.\]
Note that $\nu^4=1$ and $\nu$ preserves $[\cdot,\cdot]$ (resp. $\pair{\cdot,\cdot}$).
Now, the automorphism $\nu$ of $\geh$ is explicitly given by
\begin{align}
\label{Gauto}\nu\left(x_{\pm\alpha_i^{(j)}}\right)&=-x_{\pm\alpha_{2l-i-j+2}^{(j)}}\hspace{0.55cm}\text{if }\alpha_i^{(j)}\text{ does not contain }\alpha_l\text{ or }\alpha_{l+1},\\
\label{Gauto2}\nu\left(x_{\pm\alpha_i^{(j)}}\right)&=\pm\zeta x_{\pm\alpha_{2l-i-j+2}^{(j)}}\quad\text{if }\alpha_i^{(j)}\text{ contains exactly one of }\alpha_l\text{ or }\alpha_{l+1},\\
\label{Gauto3}\nu\left(x_{\pm\alpha_i^{(j)}}\right)&=x_{\pm\alpha_{2l-i-j+2}^{(j)}}\hspace{0.7cm}\text{ if }\alpha_i^{(j)}\text{ contains both }\alpha_l\text{ and }\alpha_{l+1}.
\end{align}

For $j\in\Z$ set
\[\geh_{(j)}=\{x\in\geh\mid\nu x=\zeta^jx\}.\]
Now, based on (\ref{Gauto})-(\ref{Gauto3}) we obtain $\geh_{(j)}\ (j\in\Z/4\Z)$ as follows.
\begin{align*}
\geh_{(0)}=&\bigoplus_{i=1}^l\left(\C x_{\alpha_i^{(2l-2i+2)}}\op\C(\alpha_i+\alpha_{2l-i+1})\op\C x_{-\alpha_i^{(2l-2i+2)}}\right)\op\bigoplus_{\chi=\pm1}\bigoplus_{j=1}^{l-1}\bigoplus_{i=1}^{l-j}\C\left(x_{\chi\alpha_i^{(j)}}-x_{\chi\alpha_{2l-i-j+2}^{(j)}}\right)\\
&\op\bigoplus_{\chi=\pm1}\bigoplus_{j=1}^{l-1}\bigoplus_{i=j+1}^l\C\left(x_{\chi\alpha_i^{(2l-i-j+2)}}+x_{\chi\alpha_j^{(2l-i-j)}}\right),\\
\geh_{(1)}=&\bigoplus_{i=1}^l\left(\C\left(x_{\alpha_i^{(l-i+1)}}+x_{\alpha_{l+1}^{(l-i+1)}}\right)\op\C\left(x_{-\alpha_i^{(l-i+1)}}-x_{-\alpha_{l+1}^{(l-i+1)}}\right)\right),\\
\geh_{(2)}=&\bigoplus_{i=1}^l\C(\alpha_i-\alpha_{2l-i+1})\op\bigoplus_{\chi=\pm1}\bigoplus_{j=1}^{l-1}\bigoplus_{i=1}^{l-j}\C\left(x_{\chi\alpha_i^{(j)}}+x_{\chi\alpha_{2l-i-j+2}^{(j)}}\right)\\
&\op\bigoplus_{\chi=\pm1}\bigoplus_{j=1}^{l-1}\bigoplus_{i=j+1}^l\C\left(x_{\chi\alpha_i^{(2l-i-j+2)}}-x_{\chi\alpha_j^{(2l-i-j)}}\right),\\
\geh_{(3)}=&\bigoplus_{i=1}^l\left(\C\left(x_{\alpha_i^{(l-i+1)}}-x_{\alpha_{l+1}^{(l-i+1)}}\right)\op\C\left(x_{-\alpha_i^{(l-i+1)}}+x_{-\alpha_{l+1}^{(l-i+1)}}\right)\right).
\end{align*}
The twisted affine Lie algebra $\hat{\geh}[\nu]$ associated to $\geh$ and $\nu$ is given by
\[\hat{\geh}[\nu]=\bigoplus_{m\in\frac{1}{4}\Z}\geh_{(4m)}\ot t^m\oplus\C c\]
with Lie bracket
\[[x\ot t^m,y\ot t^n]=[x,y]\ot t^{m+n}+\pair{x,y}m\delta_{m+n,0}c,\quad [\hat{\geh}[\nu],c]=0\]
for $m,n\in\frac{1}{4}\Z$, $x\in\geh_{(4m)}$ and $y\in\geh_{(4n)}$.
We also define the Lie algebra $\tilde{\geh}[\nu]$ by
\[\tilde{\geh}[\nu]=\hat{\geh}[\nu]\oplus\C d,\]
where $d$ is the degree operator such that
\[[d,x\ot t^m]=mx\ot t^m\]
for $m\in\frac{1}{4}\Z$, $x\in\geh_{(4m)}$ and $[d,c]=0$.
This Lie algebra $\hat{\geh}[\nu]$ ( or $\tilde{\geh}[\nu]$) is isomorphic to the twisted affine Lie algebra of type $A_{2l}^{(2)}$.
See Table \ref{Dynkin} for its Dynkin diagram.

\begin{table}[h]
\begin{align*}
\xymatrix{*{\underset{0}{\bigcirc}}\ar@<0.8ex>@{<=}[r]&*{\underset{1}{\bigcirc}}\ar@<0.8ex>@{-}[r]&*{\underset{2}{\bigcirc}}\ar@<0.8ex>@{--}[r]&*{\underset{l-2}{\bigcirc}}\ar@<0.8ex>@{-}[r]&*{\underset{l-1}{\bigcirc}}\ar@<0.8ex>@{<=}[r]&*{\underset{l}{\bigcirc}}}
\end{align*}
\caption{Twisted affine Dynkin diagram of type $A_{2l}^{(2)}$\label{Dynkin}}
\end{table}

\begin{theorem}\label{rep}
$($\cite[Theorem 3.1]{CLM}, \cite[Theorem 3]{FLM},\cite[Theorem 9.1]{L}$)$ The representation of $\hat{\heh}[\nu]$ on $V_L^T$ extends uniquely to a Lie algebra representation of $\tilde{\geh}[\nu]$ on $V_L^T$ such that
\[(x_\alpha)_{(4m)}\ot t^m\mapsto Y_\alpha^{\hat{\nu}}(m)\]
for all $m\in\frac{1}{4}\Z$ and $\alpha\in L_2=\{\alpha\in L\mid\pair{\alpha,\alpha}=2\}$.
Moreover $V_L^T$ is irreducible as a $\tilde{\geh}[\nu]$-module.
\end{theorem}
Therefore $V_L^T$ is an integrable highest weight module of highest weight $\La_0$, where $\La_0$ is the fundamental weight such that $\pair{\La_0,c}=1$ and $\pair{\La_0,\heh_{(0)}}=0=\pair{\La_0,d}$.
A highest weight vector is ${\bf 1}_T$.

\subsection{Standard module}
From Theorem \ref{rep}, we have the basic $\tilde{\geh}[\nu]$-module $L(\La_0)\simeq V_L^T$.
We also have the following formulas on $V_L^T$. 
\begin{align}
\label{ede}&e_\alpha de_\alpha^{-1}=d+\alpha-\frac{1}{8}\pair{\sum_{j=0}^3\nu^j\alpha,\alpha}c,\\
\label{ehe}&e_\alpha he_\alpha^{-1}=h-\alpha(h)c,\\
\label{ehme}&e_\alpha h(m)e_\alpha^{-1}=h(m)\text{ for }j\neq0,\\
\label{exe}&e_\alpha Y_\beta^{\hat{\nu}}(m)e_\alpha^{-1}=C(\alpha,\beta)Y_\beta^{\hat{\nu}}(m-\pair{\alpha,\beta_{(0)}}).
\end{align}

Next we consider the standard $\tilde{\geh}[\nu]$-module $L(k\La_0)$ of higher level $k$, namely the integrable highest weight $\tilde{\geh}[\nu]$-module of highest weight $k\La_0$.
Since we have $L(\La_0)\simeq V_L^T$, we can realize $L(k\La_0)$ as a submodule of the tensor product of $k$ copies of $V_L^T$ as
\[L(k\La_0)\simeq U(\tilde{\geh}[\nu])\cdot v_0\subset(V_L^T)^{\ot k},\]
where $v_0={\bf 1}_T\ot\cdots\ot{\bf 1}_T$ is a highest weight vector of $L(k\La_0)$.
The action of $\tilde{\geh}[\nu]$ on $L(k\La_0)$ is given by coproduct
\[\Delta^{(k-1)}(x)=x\ot1\ot\cdots\ot1+1\ot x\ot\cdots\ot1+\cdots+1\ot1\ot\cdots\ot x,\]
where each term has $k$ components.
It is also true for the vertex operators.
For a positive integer $n$, we set
\begin{align}\label{TVO3}
x_{n\alpha}^{\hat{\nu}}(z)=[\Delta^{(k-1)}(Y^{\hat{\nu}}(\iota(e_\alpha),z))]^n.
\end{align}
Note that $x_{(k+1)\alpha}^{\hat{\nu}}(z)=0$ because of the null identity on $V_L$ such that $e^{\Delta_z}(\iota(e_\alpha)^2)\cdot {\bf 1}=0$.
We also define a component operator $x_{n\alpha}^{\hat{\nu}}(m)$ by
\begin{align}\label{component op}
x_{n\alpha}^{\hat{\nu}}(z)=\sum_{m\in\frac{1}{4}\Z}x_{n\alpha}^{\hat{\nu}}(m)z^{-m-n\pair{\alpha,\alpha}/2}.
\end{align}

We recall the Cartan subalgebra and simple roots for twisted affine Lie algebras introduced in \cite[\S8.3]{K}.
The Cartan subalhgebra is obtained by
\[\heh_{(0)}\oplus\C c\oplus\C d.\]
Chevalley generators $h_i$ $(0\le i\leq l)$ are given as follows.
\[\begin{cases}
h_0=-\sum_{i=1}^{2l}\alpha_i\\
h_i=\alpha_i+\alpha_{2l-i+1}&\text{for }i=1,\ldots,l-1\\
h_l=2(\alpha_l+\alpha_{l+1}).
\end{cases}\]
We know that $\tilde{\geh}[\nu]$ contains a finite-dimensional simple Lie algebra $\geh_{(0)}$ (It coincides with the notation $\geh_{\bar{0}}$ used in \cite{K}).
Note that the Dynkin diagram of $\geh_{(0)}$ is of type $B_l$ .
One can take the set of simple roots of $\geh_{(0)}$ as that of the image under the projection $P_0$ to $\heh_{(0)}$.
Thus we can take it as $\{(\alpha_1)_{(0)},\ldots,(\alpha_l)_{(0)}\}$.

Set $Q=\bigoplus_{i=1}^l\Z\alpha_i\subset L$.
Note that $Q_{(0)}=P_0Q$ should be understood as the root lattice of $\tilde{\geh}[\nu]$.
Following \cite{OT}, we define the adjoint action of $e_\alpha$ on $\tilde{\geh}[\nu]$ by.
\begin{align}
\label{ece}&e_\alpha ce_\alpha^{-1}=c,\\
\label{ede2}&e_\alpha de_\alpha^{-1}=d+\alpha-\frac{1}{2}\pair{\alpha_{(0)},\alpha_{(0)}}c,\\
\label{ehe2}&e_\alpha he_\alpha^{-1}=h-\alpha(h)c\text{ for }h\in\heh_{(0)},\\
\label{ehme2}&e_\alpha h(m)e_\alpha^{-1}=h(m)\text{ for }m\neq0,\\
\label{exe2}&e_\alpha x_\beta^{\hat{\nu}}(m)e_\alpha^{-1}=C(\alpha,\beta)x_\beta^{\hat{\nu}}(m-\pair{\alpha,\beta_{(0)}}).
\end{align}
We have its action on $(V_L^T)^{\ot k}$ by $e_\alpha\mapsto e_\alpha\ot\cdots\ot e_\alpha$, and hence also on $L(k\La_0)$.
These calculations are based on (\ref{ede})-(\ref{exe}) on $V_L^T$.
This action corresponds to the translation operator of the affine Weyl group of $\tilde{\geh}[\nu]$.
See Section 1.5 of \cite{BKP} for untwisted cases.

Next lemma reveals the relation between the twisted vertex operators for positive and negative roots.
The proof is completely parallel to \cite[Theorem 5.6]{LP}, \cite[Theorem 6.4]{P} or \cite[Lemma 3]{OT}.
\begin{lemma}\label{TVO formula}
We renormalize the twisted vertex operator $x_\alpha^{\hat{\nu}}(z)$ as $\tilde{x}_\alpha^{\hat{\nu}}(z)=4\sigma(\alpha)^{-1}x_\alpha^{\hat{\nu}}(z)$ for $\alpha\in L_2$.
Then, for $p,q\geq0$ such that $p+q=k$,  we have
\begin{align}\label{TVO formula2}
\frac{1}{p!}E^-(\alpha,z)(z\tilde{x}_\alpha^{\hat{\nu}}(z))^pE^+(\alpha,z)=\frac{1}{q!}\ep_C(\alpha,-\alpha)^{-q}(z\tilde{x}_{-\alpha}^{\hat{\nu}}(z))^qe_\alpha z^{\alpha_{(0)}+\frac{k\pair{\alpha_{(0)},\alpha_{(0)}}}{2}}
\end{align}
as an operator on $(V_L^T)^{\ot k}$ or $L(k\La_0)$.
In particular, (\ref{TVO formula2}) can be rewritten as
\begin{align}\label{TVO formula3}
E^-(\alpha,z){\rm exp}(z\tilde{x}_\alpha^{\hat{\nu}}(z))E^+(\alpha,z)={\rm exp}(\ep_C(\alpha,\alpha)^{-1}z\tilde{x}_{-\alpha}^{\hat{\nu}}(z))e_\alpha z^{\alpha_{(0)}+\frac{k\pair{\alpha_{(0)},\alpha_{(0)}}}{2}}.
\end{align}
\end{lemma}

\section{principal subspace}

Let $\Delta_+$ be the set of positive roots.
We set
\[\enu=\bigoplus_{\alpha\in\Delta_+}\C x_\alpha.\]
The algebra $\enu$ is the nilradical of the Borel subalgebra.
Consider its twisted affinization
\[\hat{\enu}[\nu]=\bigoplus_{m\in\frac{1}{4}\Z}\enu_{(4m)}\ot t^m\oplus\C c\]
and its subalgebra
\[\overline{\enu}[\nu]=\bigoplus_{m\in\frac{1}{4}\Z}\enu_{(4m)}\ot t^m.\]
In \cite{CLM, CMP, FS, FS2}, the principal subspace $W(k\La_0)$ of $L(k\La_0)$ is defined by
\[W(k\La_0)=U(\overline{\enu}[\nu])\cdot v_0.\]

\subsection{Twisted quasi-particle}
In this section, we introduce the twisted quasi-particle and its monomials.
We define the twisted quasi-particle of color $i$, charge $n$ and energy $-m$ for each simple root $\alpha_i$, $n\in\mathbb{N}$ and $m\in\frac{1}{4}\Z$ as the coefficient $x_{n\alpha_i}^{\hat{\nu}}(m)$ in (\ref{component op}).

From \cite{BS}, we review twisted quasi-particle monomials.
A twisted quasi-particle monomial is defined by
\begin{align}\label{QP}
\nonumber b&=b(\alpha_l)\cdots b(\alpha_1)\\
&=x_{n_{r_l^{(1)},l}\alpha_l}^{\hat{\nu}}(m_{r_l^{(1)},l})\cdots x_{n_{1,l}\alpha_l}^{\hat{\nu}}(m_{1,l})\cdots x_{n_{r_1^{(1)},1}\alpha_1}^{\hat{\nu}}(m_{r_1^{(1)},1})\cdots x_{n_{1,1}\alpha_1}^{\hat{\nu}}(m_{1,1})
\end{align}
with $1\leq n_{r_i^{(1)},i}\leq\cdots\leq n_{1,i}$, $m_{r_i^{(1)},i}\leq\cdots\leq m_{1,i}$ for each $i$.
For (\ref{QP}), we set the sequences $\calR^\prime,\calE$ as
\[\mathcal{R}^\prime=\left(n_{r_l^{(1)},l},\ldots,n_{1,i};\cdots;n_{r_1^{(1)},1},\ldots,n_{1,1}\right),\quad\mathcal{E}=\left(m_{r_l^{(1)},l},\ldots,m_{1,1};\cdots;m_{r_1^{(1)},1},\ldots,m_{1,1}\right).\]
They are called charge-type, energy-type respectively.
For these charge-type and energy-type, we define total-charge and total-energy by
\[{\rm chg}\ b=\sum_{i=1}^l\sum_{p=1}^{r_i^{(1)}}n_{p,i},\quad{\rm en}\ b=\sum_{i=1}^l\sum_{p=1}^{r_i^{(1)}}m_{p,i}.\]
The dual-charge-type
\[\mathcal{R}=\left(r_l^{(1)},\ldots,r_l^{(s_l)};\cdots;r_1^{(1)},\ldots,r_1^{(s_1)}\right)\]
is defined in the way that $(r_i^{(1)},\ldots,r_i^{(s_i)})$ is the transposed partition of $(n_{1,i},\ldots,n_{r_i^{(1)},i})$ for each $i$.
We also define the $color$-$type$ $\calC$ by
\[\mathcal{C}=\left(r_l,\ldots,r_1\right)\]
where $r_i,n_{p,i},r_i^{(t)}$ are related as
\begin{align}\label{color}
r_i=\sum_{p=1}^{r_i^{(1)}}n_{p,i}=\sum_{t=1}^{s_i}r_i^{(t)}.
\end{align}
We denote the set of all quasi-particle monomials of the form (\ref{QP}) by $M_{QP}$.
In this definition, $r_i^{(s)}$ stands for the number of quasi-particles of color $i$ and charge greater then or equal to $s$ in the monomial $b$.
Remark that $r_i^{(s)}=0$ for $k<s$ since we have $x_{(k+1)\alpha}^{\hat{\nu}}(z)=0$ on $L(k\La_0)$ for $\alpha\in L$.
The charge-type and the dual-charge-type of (\ref{QP}) can be replaced by the $l$-tuple of Young diagrams, so that $i$-th diagram correspond to $b(\alpha_i)$.
That is, the $i$-th diagram has $r_i^{(s)}$ boxes in the $s$-th row and $n_{p,i}$ boxes in the $p$-th column, where rows are counted from the top and columns are counted from the left.
See Table \ref{Young}.
Conversely, we are able to recover twisted quasi-particle monomial with charge-type $\mathcal{R}^\prime$ from such Young diagrams.

\begin{table}[h]
\begin{picture}(400,150)(-20,0)
\put(100,20){\line(0,1){95}}
\put(100,115){\line(1,0){159}}
\put(100,20){\line(1,0){59}}
\put(118,20){\line(0,1){95}}
\put(141,20){\line(0,1){95}}
\put(159,20){\line(0,1){95}}
\multiput(100,38)(0,18){2}{\line(1,0){59}}
\put(100,79){\line(1,0){100}}
\put(100,97){\line(1,0){159}}
\put(182,79){\line(0,1){36}}
\put(200,79){\line(0,1){36}}
\put(218,97){\line(0,1){18}}
\multiput(241,97)(18,0){2}{\line(0,1){18}}
\multiput(162,59)(4,4){5}{\circle*{0.7}}
\multiput(123.3,29)(6.2,0){3}{\circle*{0.6}}
\multiput(123.3,47)(6.2,0){3}{\circle*{0.6}}
\multiput(123.3,67.5)(6.2,0){3}{\circle*{0.6}}
\multiput(123.3,88)(6.2,0){3}{\circle*{0.6}}
\multiput(123.3,106)(6.2,0){3}{\circle*{0.6}}
\multiput(164.3,88)(6.2,0){3}{\circle*{0.6}}
\multiput(164.3,106)(6.2,0){3}{\circle*{0.6}}
\multiput(223.5,106)(6.2,0){3}{\circle*{0.6}}
\multiput(109,61.3)(0,6.2){3}{\circle*{0.6}}
\multiput(150,61.3)(0,6.2){3}{\circle*{0.6}}
\put(85,35){\makebox(0,0)[t]{$r_i^{(s_i)}$}}
\put(85,53){\makebox(0,0)[t]{$r_i^{(s_i-1)}$}}
\put(85,94){\makebox(0,0)[t]{$r_i^{(2)}$}}
\put(85,112){\makebox(0,0)[t]{$r_i^{(1)}$}}
\put(109,130){\makebox(0,0)[t]{$n_{1,i}$}}
\put(150,130){\makebox(0,0)[t]{$n_{r_i^{(s_i)},i}$}}
\put(191,130){\makebox(0,0)[t]{$n_{r_i^{(2)},i}$}}
\put(250,130){\makebox(0,0)[t]{$n_{r_i^{(1)},i}$}}
\end{picture}
\caption{Young diagram for $b(\alpha_i)$ with $s_i=n_{1,i}$\label{Young}}
\end{table}

For two charge-types $\mathcal{R}^\prime$ and $\overline{\mathcal{R}^\prime}=\left(\overline{n}_{\overline{r}_l^{(1)},l},\ldots,\overline{n}_{1,1}\right)$, we write $\mathcal{R}^\prime<\overline{\mathcal{R}^\prime}$ if there exists $i$ and $s$ such that $r_j^{(1)}=\bar{r}_j^{(1)},n_{t,j}=\bar{n}_{t,j}$ for $j<i,1\leq t\leq r_j^{(t)}$ and $n_{1,i}=\bar{n}_{1,i},n_{2,i}=\bar{n}_{2,1},\ldots,n_{s-1,i}=\bar{n}_{s-1,i},n_{s,i}<\bar{n}_{s,i}$ or $n_{t,i}=\bar{n}_{t,i}$ for $1\leq t\leq r_i^{(1)},r_i^{(1)}<\bar{r}_i^{(1)}$.
We apply this order $"<"$ to energy-types and color-types.
For two monomials $b$ and $\overline{b}$ with charge-types $\mathcal{R}^\prime$ and $\overline{\mathcal{R}^\prime}$, energy-type $\mathcal{E}$ and $\overline{\mathcal{E}^\prime}$, respectively, we write $b<\overline{b}$ if
\begin{enumerate}
\item$\mathcal{R}^\prime<\overline{\mathcal{R}^\prime}\text{ or}$\\
\item$\mathcal{R}^\prime=\overline{\mathcal{R}^\prime}\text{ and }\mathcal{E}<\overline{\mathcal{E}^\prime}$.
\end{enumerate}
By \cite{BKP, BS}, we have the linear order $"<"$ on the set of monomials in $M_{QP}$.

Next, we consider relations among twisted quasi-particles which give the basis conditions that will be described later.
The following lemma is proved in the same way as \cite[Lemma 3.2]{CLM} and is a special case of \cite[Lemma 2.1, Lemma 2.2]{CMP}.

\begin{lemma}\label{component rel1}
{\phantom{a}}
\begin{enumerate}
\renewcommand{\labelenumi}{\roman{enumi}}
\item[(i)]$Y_{\alpha_i}^{\hat{\nu}}(m)=0$\quad {\rm for} $i\notin\{l,l+1\}$ {\rm and} $m\in\frac{1}{4}+\frac{1}{2}\Z$\\
\item[(ii)]$Y_{\alpha_l}^{\hat{\nu}}(m)=Y_{\alpha_{l+1}}^{\hat{\nu}}(m)=0$\quad {\rm for} $m\in\frac{1}{2}\Z$\\
\item[(iii)]$Y_{\alpha_{2l-i+1}}^{\hat{\nu}}(m)=(-1)^{2m+1}Y_{\alpha_i}^{\hat{\nu}}(m)$\quad {\rm for} $i\notin\{l,l+1\}$ {\rm and} $m\in\frac{1}{2}\Z$\\
\item[(iv)]$Y_{\alpha_{l+1}}^{\hat{\nu}}(m)=\zeta^{4m-1}Y_{\alpha_l}^{\hat{\nu}}(m)$\quad {\rm for} $m\in\frac{1}{4}+\frac{1}{2}\Z$
\end{enumerate}
\end{lemma}

\begin{proof}
We show that (i) holds for $\alpha_i$ with $i\notin\{l,l+1\}$.
By taking $j=2$ and $v=\iota(e_{\alpha_i})$ in (\ref{TVO rel}) and using Proposition \ref{CMP prop}, we have
\begin{align}\label{TVO eq}
Y^{\hat{\nu}}(\iota(e_{\alpha_i}),z)=Y^{\hat{\nu}}(\iota(e_{\alpha_i}),z)\mid_{z^\frac{1}{4}\rightarrow-z^\frac{1}{4}}.
\end{align}
As the component operator, we have
\begin{align}\label{com eq}
\sum_{m\in\frac{1}{4}\Z}Y_{\alpha_i}^{\hat{\nu}}(m)z^{-m-1}=\sum_{m\in\frac{1}{4}\Z}(-1)^{4m}Y_{\alpha_i}^{\hat{\nu}}(m)z^{-m-1}.
\end{align}
Thus we obtain the statement by taking the residue of the both sides of (\ref{com eq}).
(ii) is obtained from the same argument.

Then we take $j=1$ and $v=\iota(e_{\alpha_i})$ $(i\neq l,l+1)$ in (\ref{TVO rel}).
From Proposition \ref{CMP prop}, we have
\[Y^{\hat{\nu}}(-\iota(e_{\alpha_{2l-i+1}}),z)=Y^{\hat{\nu}}(\iota(e_{\alpha_i}),z)\mid_{z^\frac{1}{4}\rightarrow-\zeta z^\frac{1}{4}}.\]
Thus we obtain
\[-\sum_{m\in\frac{1}{4}\Z}Y_{\alpha_{2l-i+1}}^{\hat{\nu}}(m)z^{-m-1}=\sum_{m\in\frac{1}{4}\Z}\zeta^{4m}Y_{\alpha_i}^{\hat{\nu}}(m)z^{-m-1}.\]
This implies (iii).
We are able to prove (iv) in the same argument by taking $j=1$ and $v=\iota(e_{\alpha_l})$ in (\ref{TVO rel}).
\end{proof}

For every simple root $\alpha_i$, consider the one-dimensional subalgebra of $\geh$
\[\enu_{\alpha_i}=\C x_{\alpha_i},\]
and its affinizations
\[\overline{\enu}_{\alpha_i}[\nu]=\bigoplus_{m\in\frac{1}{4}\Z}{(\enu_{\alpha_i})}_{(4m)}\ot t^m.\]
Set
\[U=U(\bar{\enu}_{\alpha_l}[\nu])\cdots U(\bar{\enu}_{\alpha_1}[\nu])\]
to be the subspace of $U(\tilde{\geh}[\nu])$.
From Lemma \ref{component rel1}, we can prove the following lemma by arguing as in \cite[Lemma 3.1]{G}.

\begin{lemma}\label{ordered set}
$W(k\La_0)=U(\bar{\enu}[\nu])\cdot v_0=U\cdot v_0$
\end{lemma}

\begin{proof}
Lemma \ref{component rel1} implies that the subalgebras $U(\bar{\enu}_{\alpha_1}[\nu]),\ldots,U(\bar{\enu}_{\alpha_l}[\nu])$ span the principal subspace $W(k\La_0)$.
Then the statement that the ordered set $U$ spans $W(k\La_0)$ is showed as in the proof of \cite[Lemma 3.1]{G}.
\end{proof}

From the definition of the twisted vertex operators on $L(k\La_0)$, we have the following lemma. 

\begin{lemma}\label{TVO rel2}
For fixed $1\leq n_2\leq n_1$, $N=0,1\ldots,2n_2-1$ and $i=1,\ldots,l$, we have
\[\left(\left(\frac{d}{dz}\right)^Nx_{n_2\alpha_i}^{\hat{\nu}}(z)\right)x_{n_1\alpha_i}^{\hat{\nu}}(z)=A_N(z)x_{(n_1+1)\alpha_i}^{\hat{\nu}}(z)+B_N(z)\frac{d}{dz}x_{(n_1+1)\alpha_i}^{\hat{\nu}}(z)\]
where $A_N(z)$ and $B_N(z)$ are some formal series with coefficients in the set of quasi-particle polynomials.
\end{lemma}

\begin{proof}
The relation (\ref{TVO3}) implies that we have
\begin{align}\label{TVO eq2}
(n+1)\left(\frac{d}{dz}x_{\alpha_i}^{\hat{\nu}}(z)\right)x_{n\alpha_i}^{\hat{\nu}}(z)=\frac{d}{dz}x_{(n+1)\alpha_i}^{\hat{\nu}}(z).
\end{align}
By using the Leibniz rule and (\ref{TVO3}), we have
\[\left(\frac{d}{dz}\right)^Nx_{n_2\alpha_i}(z)=\sum_{\substack{k_1+\cdots+k_{n_2}=N\\
0\leq k_1,\ldots,k_{n_2}}}\binom{N}{k_1\cdots k_{n_2}}\left(\frac{d}{dz}\right)^{k_1}x_{\alpha_i}^{\hat{\nu}}(z)\cdots\left(\frac{d}{dz}\right)^{k_{n_2}}x_{\alpha_i}^{\hat{\nu}}(z).\]
For $N=0,1,\ldots,2n_2-1$, we obtain the statement by combining the above formula, (\ref{TVO3}) and (\ref{TVO eq2}).
\end{proof}

We set
\begin{align}\label{rho}
\rho_i=\frac{1}{2}\pair{(\alpha_i)_{(0)},(\alpha_i)_{(0)}}.
\end{align}
By combining Lemma \ref{component rel1}, \ref{TVO rel2}, the next lemma can be proved in the same way as Lemma 4.1 in \cite{BS}. 

\begin{lemma}\label{component rel2}
Let $1\leq n_2\leq n_1$ be fixed.
For $i=1,\ldots,l$ and fixed $j\in \rho_in_2+\frac{1}{2}\Z,M\in\rho_i(n_1+n_2)+\frac{1}{2}\Z$, the $4\rho_in_2$ monomials from the set
\[A=\{x_{n_2\alpha_i}^{\hat{\nu}}(j)x_{n_1\alpha_i}^{\hat{\nu}}(M-j),x_{n_2\alpha_i}^{\hat{\nu}}(j-\frac{1}{2})x_{n_1\alpha_i}^{\hat{\nu}}(M-j+\frac{1}{2}),\ldots\]
\[\ldots,x_{n_2\alpha_i}^{\hat{\nu}}(j-\frac{4\rho_in_2-1}{2})x_{n_1\alpha_i}^{\hat{\nu}}(M-j+\frac{4\rho_in_2-1}{2})\}\]
can be expressed as a linear combination of monomials from the set
\[\left\{x_{n_2\alpha_i}^{\hat{\nu}}(s)x_{n_1\alpha_i}^{\hat{\nu}}(t)\mid s+t=M\right\}\setminus A\]
and monomials which have as a factor the quasi-particle $x_{(n_1+1)\alpha_i}^{\hat{\nu}}(j^\prime)$ for $j^\prime\in\rho_i(n_1+1)+\frac{1}{2}\Z$.\\
\end{lemma}

\begin{proof}
For $N=0,1,\ldots,4\rho_in_2-1$, we consider the expansion of formal series
\begin{align}\label{expansion}
\frac{1}{N!}\left(\left(\frac{d}{dz}\right)^Nx_{n_2\alpha_i}^{\hat{\nu}}(z)\right)x_{n_1\alpha_i}^{\hat{\nu}}(z)=\sum_{\substack{M\in\rho_i(n_1+n_2)+\frac{1}{2}\Z\\s+t=M}}\binom{-s-n_2}{N}x_{n_2\alpha_i}^{\hat{\nu}}(s)x_{n_1\alpha_i}^{\hat{\nu}}(t)z^{-M-n_2-n_1-N}.
\end{align}
Note that the energy $m$ for quasi-particle with color $i$ and charge $n$ goes over $\rho_in+\frac{1}{2}\Z$ because of Lemma \ref{component rel1}.
Then for fixed $j\in \rho_in_2+\frac{1}{2}\Z, M\in\rho_i(n_1+n_2)+\frac{1}{2}\Z$, we can separate the right hand side of (\ref{expansion}) by
\begin{align}\label{expansion2}
\nonumber\binom{-j-n_2}{N}x_{n_2\alpha_i}^{\hat{\nu}}(j)x_{n_1\alpha_i}^{\hat{\nu}}(M-j)+\binom{-j-n_2+\frac{1}{2}}{N}x_{n_2\alpha_i}^{\hat{\nu}}\left(j-\frac{1}{2}\right)x_{n_1\alpha_i}^{\hat{\nu}}\left(M-j+\frac{1}{2}\right)+\cdots\\
+\binom{-j-n_2+\frac{4\rho_in_2-1}{2}}{N}x_{n_2\alpha_i}^{\hat{\nu}}\left(j-\frac{4\rho_in_2-1}{2}\right)x_{n_1\alpha_i}^{\hat{\nu}}\left(M-j+\frac{4\rho_in_2-1}{2}\right)+\text{ other terms}.
\end{align}
From Lemma \ref{TVO rel2}, we also have
\begin{align}\label{expansion3}
(\ref{expansion2})=A_N(z)x_{(n_1+1)\alpha_i}^{\hat{\nu}}(z)+B_N(z)\frac{d}{dz}x_{(n_1+1)\alpha_i}^{\hat{\nu}}(z).
\end{align}
Thus we obtain the linear equation
\begin{align*}
\begin{pmatrix}
\binom{-j-n_2}{0}&\binom{-j-n_2+\frac{1}{2}}{0}&\cdots&\binom{-j-n_2+\frac{4\rho_in_2-1}{2}}{0}\\
\binom{-j-n_2}{1}&\binom{-j-n_2+\frac{1}{2}}{1}&\cdots&\binom{-j-n_2+\frac{4\rho_in_2-1}{2}}{1}\\
\vdots&\vdots&\ddots&\vdots\\
\binom{-j-n_2}{4\rho_in_2-1}&\binom{-j-n_2+\frac{1}{2}}{4\rho_in_2-1}&\cdots&\binom{-j-n_2+\frac{4\rho_in_2-1}{2}}{4\rho_in_2-1}
\end{pmatrix}
\begin{pmatrix}
x_{n_2\alpha_i}^{\hat{\nu}}(j)x_{n_1\alpha_i}^{\hat{\nu}}(M-j)\\
x_{n_2\alpha_i}^{\hat{\nu}}\left(j-\frac{1}{2}\right)x_{n_1\alpha_i}^{\hat{\nu}}\left(M-j+\frac{1}{2}\right)\\
\vdots\\
x_{n_2\alpha_i}^{\hat{\nu}}\left(j-\frac{4\rho_in_2-1}{2}\right)x_{n_1\alpha_i}^{\hat{\nu}}\left(M-j+\frac{4\rho_in_2-1}{2}\right)
\end{pmatrix}=X,
\end{align*}
where $X$ is a vector whose components are linear combinations of monomials from the set $\{x_{n_2\alpha_i}^{\hat{\nu}}(s)x_{n_1\alpha_i}^{\hat{\nu}}(t)\mid s+t=M\}\setminus A$, and monomials which have as a factor the quasi-particle $x_{(n_1+1)\alpha_i}^{\hat{\nu}}(j^\prime)$ for $j^\prime\in\rho_i(n_1+1)$.
Now, the statement follows from the regularity of the coefficient matrix and it can be proved in the same way as Lemma 4.1 in \cite{BS}.
\end{proof}

Let $b$ be a quasi-particle monomial in $M_{QP}$.
From Lemma \ref{component rel2}, the $4\rho_in_2$ monomials
\[x_{n_2\alpha_i}^{\hat{\nu}}(m)x_{n_1\alpha_i}^{\hat{\nu}}(m^\prime)v,x_{n_2\alpha_i}^{\hat{\nu}}\left(m-\frac{1}{2}\right)x_{n_1\alpha_i}^{\hat{\nu}}\left(m^\prime+\frac{1}{2}\right)v,\ldots\]
\[\ldots,x_{n_2\alpha_i}^{\hat{\nu}}\left(m-\frac{4\rho_in_2-1}{2}\right)x_{n_1\alpha_i}^{\hat{\nu}}\left(m^\prime+\frac{4\rho_in_2-1}{2}\right)v\]
such that $n_2<n_1$ can be expressed as a linear combination of the monomials
\[x_{n_2\alpha_i}^{\hat{\nu}}(s)x_{n_1\alpha_i}^{\hat{\nu}}(t)v\text{ with }s\leq m-2\rho_in_2,\quad t\geq m+2\rho_in_2\]
and monomials which have a factor quasi-particle $x_{(n_1+1)\alpha_i}^{\hat{\nu}}(j)$, $j\in\rho_i(n_1+1)+\frac{1}{2}\Z$.
If $n_2=n_1$, then the monomials
\[x_{n_1\alpha_i}^{\hat{\nu}}(m)x_{n_1\alpha_i}^{\hat{\nu}}(m^\prime)v\text{ such that }m^\prime-\rho_in_1<m\leq m^\prime\]
can be expressed as a linear combination of the monomials
\[x_{n_1\alpha_i}^{\hat{\nu}}(s)x_{n_1\alpha_i}^{\hat{\nu}}(t)v\text{ with }s\leq t-\rho_in_1\]
and monomials which have a factor quasi-particle $x_{(n_1+1)\alpha_i}^{\hat{\nu}}(j)$, $j\in\rho_i(n_1+1)+\frac{1}{2}\Z$.

The next lemma reveals the relations among quasi-particle monomials with different colors.

\begin{lemma}\label{component rel3}
Let $P(z_{r_l^{(1)},l},\ldots,z_{1,1})$ be the polynomial defined by
\begin{align*}
P(z_{r_l^{(1)},l},\ldots,z_{1,1})&=\prod_{i=1}^l\prod_{p=1}^{r_i^{(1)}}\prod_{q=1}^{r_{i-1}^{(1)}}\prod_{j=0}^3\left(1-\zeta^j\frac{z_{q,i-1}^\frac{1}{4}}{z_{p,i}^\frac{1}{4}}\right)^{-\pair{\nu^j\alpha_i,\alpha_{i-1}}{\rm min}\{n_{p,i},n_{q,i-1}\}}\\
&=\prod_{i=1}^l\prod_{p=1}^{r_i^{(1)}}\prod_{q=1}^{r_{i-1}^{(1)}}\left(1-\frac{z_{q,i-1}^\frac{1}{2}}{z_{p,i}^\frac{1}{2}}\right)^{{\rm min}\{n_{p,i},n_{q,i-1}\}}.
\end{align*}
Then we have
\[P(z_{r_l^{(1)},l},\ldots,z_{1,1})\prod_{1\leq s<t\leq r_l^{(1)}}\left(1+\frac{z_{s,l}^\frac{1}{2}}{z_{t,l}^\frac{1}{2}}\right)^{n_{t,l}}x_{n_{r_l^{(1)},l}\alpha_l}^{\hat{\nu}}(z_{r_l^{(1)},l})\cdots x_{n_{1,1}\alpha_1}^{\hat{\nu}}(z_{1,1})v_0\]
\[\in\left[\left(\prod_{i=1}^l\prod_{p=1}^{r_i^{(1)}}z_{p,i}^{-\frac{1}{2}\sum_{q=1}^{r_{i-1}^{(1)}}{\rm min}\{n_{p,i},n_{q,i-1}\}}\right)\left(\prod_{t=1}^{r_l^{(1)}}z_{t,l}^{-\frac{1}{2}(t-1)n_{t,l}}\right)\right]W(k\La_0)[[z_{r_l^{(1)},l}^\frac{1}{4},\ldots,z_{1,1}^\frac{1}{4}]].\]
\end{lemma}

\begin{proof}
By using (\ref{Ecom}), we have
\begin{align*}
&E^-(\alpha,z)E^+(\alpha,z)e_\alpha z^{\alpha_{(0)}}E^-(\beta,w)E^+(\beta,w)e_\beta w^{\beta_{(0)}}\\
&=\prod_{j=0}^3\left(1-\zeta^j\frac{w^\frac{1}{4}}{z^\frac{1}{4}}\right)^{\pair{\nu^j\alpha,\beta}}E^-(\alpha,z)E^-(\beta,w)E^+(\alpha,z)E^+(\beta,w)e_\alpha z^{\alpha_{(0)}}e_\beta w^{\beta_{(0)}}\\
&=\prod_{j=0}^3(z^\frac{1}{4}-\zeta^jw^\frac{1}{4})^{\pair{\nu^j\alpha,\beta}}E^-(\alpha,z)E^-(\beta,w)E^+(\alpha,z)E^+(\beta,w)e_\alpha e_\beta z^{\alpha_{(0)}}w^{\beta_{(0)}}
\end{align*}
on $V_L^T$, where we recall that $z^ha=az^{h+\pair{h,\overline{a}}}$ for $h\in\heh_{(0)}$, $a\in\hat{L}_\nu$.
Therefore we have
\begin{enumerate}
\item$x_{n\alpha_i}^{\hat{\nu}}(z)x_{n^\prime\alpha_i}^{\hat{\nu}}(w)v_0\in W(k\La_0)[[z^\frac{1}{4},w^\frac{1}{4}]]$ for $i=1,\cdots,l-1$,\\
\item$(z^\frac{1}{2}+w^\frac{1}{2})^{{\rm min}\{n,n^\prime\}}x_{n\alpha_l}^{\hat{\nu}}(z)x_{n^\prime\alpha_l}^{\hat{\nu}}(w)v_0\in W(k\La_0)[[z^\frac{1}{4},w^\frac{1}{4}]]$,\\
\item$\prod_{j=0}^3(z^\frac{1}{4}-\zeta^jw^\frac{1}{4})^{-\pair{\nu^j\alpha_i,\alpha_{i-1}}{\rm min}\{n,n^\prime\}}x_{n\alpha_i}^{\hat{\nu}}(z)x_{n^\prime\alpha_{i-1}}^{\hat{\nu}}(w)v_0\in W(k\La_0)[[z^\frac{1}{4},w^\frac{1}{4}]]$\\
$\Leftrightarrow\left(1-\frac{w^\frac{1}{2}}{z^\frac{1}{2}}\right)^{{\rm min}\{n,n^\prime\}}x_{n\alpha_i}^{\hat{\nu}}(z)x_{n^\prime\alpha_{i-1}}^{\hat{\nu}}(w)v_0\in[z^{-\frac{1}{2}{\rm min}\{n,n^\prime\}}]W(k\La_0)[[z^\frac{1}{4},w^\frac{1}{4}]].$
\end{enumerate}
We can prove the lemma by applying these relations to the generating function $x_{n_{r_l^{(1)},l}\alpha_l}^{\hat{\nu}}(z_{r_l^{(1)},l})\cdots x_{n_{1,1}\alpha_1}^{\hat{\nu}}(z_{1,1})$.
\end{proof}

\subsection{Quasi-particle bases of principal subspace}
In view of Lemma \ref{component rel2} and \ref{component rel3}, we consider the following conditions ($C1$)-($C3$) for the modes $m$ in $x_{n\alpha_i}^{\hat{\nu}}(m)$ in (\ref{QP}).

\begin{align*}
(C1)\quad&m_{p,i}\in \rho_in_{p,i}+\frac{1}{2}\Z\quad\text{for }1\leq p\leq r_i^{(1)},1\leq i\leq l,\\
(C2)\quad&m_{p,i}\leq-(2p-1)\rho_in_{p,i}+\frac{1}{2}\sum_{q=1}^{r_{i-1}^{(1)}}{\rm min}\{n_{p,i},n_{q,i-1}\}\quad\text{for }1\leq p\leq r_i^{(1)},~1\leq i\leq l,\\
(C3)\quad&m_{p+1,i}\leq m_{p,i}-2\rho_in_{p,i}\quad\text{if }n_{p+1,i}=n_{p,i}\quad\text{for }1\leq p\leq r_i^{(1)}-1,1\leq i\leq l.
\end{align*}
We set
\[B_W=\bigcup_{\substack{0\leq r_1^{(k)}\leq\cdots\leq r_1^{(1)}\\
\vdots\\
0\leq r_l^{(k)}\leq\cdots\leq r_l^{(1)}}}\{b\text{ as in }(\ref{QP})\mid\text{satisfying }(C1)-(C3)\}\]
where $r_0^{(1)}=0$.
Then by using the same proof as in \cite[Theorem 4.1]{G}, we can prove the following proposition (cf. \cite{BS}).

\begin{proposition}\label{span}
The set
\[\mathcal{B}_W=\{bv_0\mid b\in B_W\}\]
spans $W(k\La_0)$.
\end{proposition}

\begin{proof}
Since we have
\[U\cdot v_0=W(k\La_0)\]
from Lemma \ref{ordered set}, we should show that every vector $bv_0$ from $U\cdot v_0$ is a linear combination of vectors from $\mathcal{B}_W$.
The condition (C1) follows from Lemma \ref{component rel1}.
Using the same argument as in \cite{G} and the fact that for fixed charge-type and total-energy, the set of quasi-particles is upper bounded with respect to our order, we can derive the remaining conditions.
All vectors $bv_0$ are expressed as a combination of vectors which satisfy the weaker condition
\[(C2)^\prime\quad m_{p,i}\leq-\pair{(\alpha_i)_{(0)},(\alpha_{i-1})_{(0)}}\sum_{q=1}^{r_{i-1}^{(1)}}{\rm min}\{n_{p,i},n_{q,i-1}\}+\frac{\delta_{i,l}}{2}(p-1)n_{p,l}-\rho_in_{p,i},\]
where the first term of the right hand side of (C2$)^\prime$ is equal to 0 when $i=1$.
This condition is a consequence of Lemma \ref{component rel3}.
In fact, for a quasi-particle $b$ which contradicts the condition (C2$)^\prime$, by taking residue of the generating function of $b$ in Lemma \ref{component rel3} we have that $b$ is a linear combination of quasi-particles $b^\prime$ larger than $b$ since the residue of the right hand side is equal to 0.
Thus by induction, we can show the claim.
From Lemma \ref{component rel2} in the case that $n_2<n_1$, we can increase an energy of quasi-particles.
Therefore  we are able to strengthen (C2$)^\prime$ to (C2) by induction.
Finally, Lemma \ref{component rel2} in the case of $n_2=n_1$ implies (C3).
\end{proof}

\subsection{Proof of linear independence}
To prove the linear independence of the set $\mathcal{B}_W$, we recall the map $\Delta^T$ introduced in \cite{CLM, Li}.
Let $\lambda_i$ $(i=1,\ldots,2l)$ be the fundamental weights of $\geh$.
Set
\[\Delta^T(\lambda_i,-z)=\zeta^{2(\lambda_i)_{(0)}}z^{(\lambda_i)_{(0)}}E^+(-\lambda_i,z).\]
Note that $\zeta^{2(\lambda_i)_{(0)}}$ and $z^{(\lambda_i)_{(0)}}$ are operators on $U_T$ and thus on $V_L^T$, and $E^+(-\lambda_i,z)$ also on $V_L^T$, so that we have
\[\Delta^T(\lambda_i,-z)\in{\rm End}\ V_L^T[[z^{\pm\frac{1}{4}}]].\]
From \cite[Proposition 3.4]{LW}, the commutation relation (\ref{Ecom}) also holds for $E^+(\lambda_i,-z)$.
The constant term of $\Delta^T(\lambda,-z)$ is denoted by $\Delta_c^T(\lambda,-z)$.
Following \cite{CLM, CMP}, we set $\theta_j:L\rightarrow\C^\times$ as the character of the root lattice $L$ to be
\[\theta_j(\alpha_j)=-\zeta,\quad\theta_j(\alpha_{2l-j+1})=\zeta,\quad\theta_j(\alpha_i)=1\text{ for }i\notin\{j,2l-j+1\}.\]
Define the Lie algebra automorphism $\tau_{\lambda_j,\theta_j}:\overline{\enu}[\nu]\rightarrow\overline{\enu}[\nu]$ by
\[\tau_{\lambda_j,\theta_j}(x_\alpha^{\hat{\nu}}(m))=\theta_j(\alpha)x_\alpha^{\hat{\nu}}(m+\pair{\alpha_{(0)},\lambda_j}).\]
For $a\in U(\overline{\enu}[\nu])$, we have
\begin{align}\label{Delta}
\Delta_c^T(\lambda_j,-z)(a\cdot{\bf 1}_T)=\tau_{\lambda_j,\theta_j}(a){\bf 1}_T
\end{align}
on $V_L^T$.
Using (\ref{Eh}), we have
\begin{align*}
[\Delta^T(\lambda_i,-z),\alpha^{\hat{\nu}}(m)]=\begin{cases}
0&(m\geq0),\\
\pair{\alpha_{(4m)},(\lambda_i)_{(-4m)}}z^m\Delta^T(\lambda_i,-z)&(m<0).
\end{cases}
\end{align*}
Therefore we obtain
\begin{align}\label{Delta2}
&[\Delta_c^T(\lambda_j,-z),\alpha^{\hat{\nu}}(m)]=0
\end{align}
 as an operator on $V_L^T$ for $\alpha^\nu(m)\in\hat{\heh}[\nu]$.
Furthermore, for a bijection map $e_{\alpha_i}$, we have
\begin{align}\label{Delta4}
[\Delta_c^T(\lambda_j,-z),e_{\alpha_i}]=0\quad(i\neq j)
\end{align}
from $z^{(\lambda_j)_{(0)}}e_{\alpha_i}=e_{\alpha_i}z^{\frac{\delta_{i,j}}{2}+(\lambda_j)_{(0)}}$ and the fact that we define $\Delta_c^T$ as the constant term.

We consider the Georgiev-type projection (cf. \cite{BS, G}).
We realize $W(k\La_0)$ as a subspace of the $k$ fold tensor product of $W(\La_0)$ in the basic module $V_L^T$.
Namely, we have
\[W(k\La_0)\subset W(\La_0)\ot\cdots\ot W(\La_0)\subset(V_L^T)^{\ot k}.\]
The projection $\pi_{\mathcal{R}}$ is defined by
\[\pi_{\mathcal{R}}:W(k\La_0)\rightarrow W(\La_0)_{(r_l^{(k)},\ldots,r_1^{(k)})}\ot\cdots\ot W(\La_0)_{(r_l^{(1)},\ldots,r_1^{(1)})}\]
for a dual-charge-type $\calR$ given by
\[\mathcal{R}=\left(r_l^{(1)},\ldots,r_l^{(k)};\cdots;r_1^{(1)},\ldots,r_1^{(k)}\right),\]
where $W(\La_0)_{(r_l^{(s)},\ldots,r_1^{(s)})}$ is the subspace of $W(\La_0)$ spanned by the vectors whose charges are $r_l^{(s)},\ldots,r_1^{(s)}$ for $1\leq s\leq k$.

By using the relation $x_{2\alpha}^{\hat{\nu}}(z)=0$ on $V_L^T$, we have
\[\pi_{\mathcal{R}}x_{n_{r_l^{(1)},l}\alpha_l}^{\hat{\nu}}(z_{r_l^{(1)},l})\cdots x_{n_{1,1}\alpha_1}^{\hat{\nu}}(z_{1,1})v_0\]
\begin{align*}
=C_{\calR}&\left(Y^{\hat{\nu}}(\iota(e_{\alpha_l}),z_{r_l^{(k)},l})^{n_{r_l^{(k)},l}^{(k)}}\cdots Y^{\hat{\nu}}(\iota(e_{\alpha_l}),z_{1,l})^{n_{1,l}^{(k)}}\cdots Y^{\hat{\nu}}(\iota(e_{\alpha_1}),z_{r_1^{(k)},1})^{n_{r_1^{(k)},1}^{(k)}}\cdots Y^{\hat{\nu}}(\iota(e_{\alpha_1}),z_{1,1})^{n_{1,1}^{(k)}}{\bf 1}_T\right.\\
&\phantom{aa}\ot\cdots\ot\\
&\phantom{aa}Y^{\hat{\nu}}(\iota(e_{\alpha_l}),z_{r_l^{(s)},l})^{n_{r_l^{(s)},l}^{(s)}}\cdots Y^{\hat{\nu}}(\iota(e_{\alpha_l}),z_{1,l})^{n_{1,l}^{(s)}}\cdots Y^{\hat{\nu}}(\iota(e_{\alpha_1}),z_{r_1^{(s)},1})^{n_{r_1^{(s)},1}^{(s)}}\cdots Y^{\hat{\nu}}(\iota(e_{\alpha_1}),z_{1,1})^{n_{1,1}^{(s)}}{\bf 1}_T\\
&\phantom{aa}\ot\cdots\ot\\
&\phantom{a.}\left.Y^{\hat{\nu}}(\iota(e_{\alpha_l}),z_{r_l^{(1)},l})^{n_{r_l^{(1)},l}^{(1)}}\cdots Y^{\hat{\nu}}(\iota(e_{\alpha_l}),z_{1,l})^{n_{1,l}^{(1)}}\cdots Y^{\hat{\nu}}(\iota(e_{\alpha_1}),z_{r_1^{(1)},1})^{n_{r_1^{(1)},1}^{(1)}}\cdots Y^{\hat{\nu}}(\iota(e_{\alpha_1}),z_{1,1})^{n_{1,1}^{(1)}}{\bf 1}_T\right)\\
\end{align*}
where $C_{\calR}$ is the constant term given by $\prod_{i=1}^l\prod_{p=1}^{r_i^{(1)}}(n_{p,i})!$ and
\[0\leq n_{p,i}^{(k)}\leq n_{p,i}^{(k-1)}\leq\cdots\leq n_{p,i}^{(2)}\leq n_{p,i}^{(1)}\leq1\text{ and }n_{p,i}=\sum_{s=1}^kn_{p,i}^{(s)}\]
for $1\leq p\leq r_i^{(1)}$, $1\leq i\leq l$.
Note that the projection of  $bv_0$ with $b$ as in (\ref{QP}) is obtained by a coefficient of this projection.

Now, we fix $s\leq k$ and consider the map
\[\Delta_c^T(\lambda_j,-z)_s=1\ot\cdots\ot1\ot\Delta_c^T(\lambda_j,-z)\ot1\ot\cdots\ot1\]
which only acts on the $s$-th tensor component.
From (\ref{Delta}), we have
\[\Delta_c^T(\lambda_j,-z)_s\pi_{\mathcal{R}}x_{n_{r_l^{(1)},l}\alpha_l}^{\hat{\nu}}(z_{r_l^{(1)},l})\cdots x_{n_{1,1}\alpha_1}^{\hat{\nu}}(z_{1,1})v_0\]
\begin{align*}
=(-\zeta)^{r_j^{(s)}}C_{\calR}&\left(Y^{\hat{\nu}}(\iota(e_{\alpha_l}),z_{r_l^{(k)}})\cdots Y^{\hat{\nu}}(\iota(e_{\alpha_l}),z_{1,l})\cdots Y^{\hat{\nu}}(\iota(e_{\alpha_1}),z_{r_1^{(k)},1})\cdots Y^{\hat{\nu}}(\iota(e_{\alpha_1}),z_{1,1}){\bf 1}_T\right.\\
&\phantom{aa}\ot\cdots\ot\\
&\left(Y^{\hat{\nu}}(\iota(e_{\alpha_l}),z_{r_l^{(s)},l})\cdots Y^{\hat{\nu}}(\iota(e_{\alpha_l}),z_{1,l})\cdots Y^{\hat{\nu}}(\iota(e_{\alpha_1}),z_{r_1^{(s)},1})\cdots Y^{\hat{\nu}}(\iota(e_{\alpha_1}),z_{1,1}){\bf 1}_T\right)\prod_{i=1}^l\prod_{p=1}^{r_i^{(s)}}z_{p,i}^\frac{\delta_{i,j}}{2}\\
&\phantom{aa}\ot\cdots\ot\\
&\phantom{a}\left.Y^{\hat{\nu}}(\iota(e_{\alpha_l}),z_{r_l^{(1)},l})\cdots Y^{\hat{\nu}}(\iota(e_{\alpha_l}),z_{1,l})\cdots Y^{\hat{\nu}}(\iota(e_{\alpha_1}),z_{r_1^{(1)},1})\cdots Y^{\hat{\nu}}(\iota(e_{\alpha_1}),z_{1,1}){\bf 1}_T\right).
\end{align*}
Thus by taking the corresponding coefficient, we have
\begin{align}\label{process}
\Delta_c^T(\lambda_j,-z)_s\pi_\mathcal{R}bv_0=(-\zeta)^{r_j^{(s)}}\pi_\mathcal{R}b^+v_0,
\end{align}
where $b^+=b(\alpha_l)\cdots b^+(\alpha_j)\cdots b(\alpha_1)$ such that
\[b^+(\alpha_j)=x_{n_{r_j^{(1)},j}\alpha_j}^{\hat{\nu}}(m_{r_j^{(1)},j})\cdots x_{n_{r_j^{(s)}+1,j}\alpha_j}^{\hat{\nu}}(m_{r_j^{(s)}+1,j})x_{n_{r_j^{(s)},j}\alpha_j}^{\hat{\nu}}\left(m_{r_j^{(s)},j}+\frac{1}{2}\right)\cdots x_{n_{1,j}\alpha_j}^{\hat{\nu}}\left(m_{1,j}+\frac{1}{2}\right).\]
For $b$ as in (\ref{QP}), we set $s=n_{1,1}$, $d=-2m_{1,1}-s$.
From  (\ref{process}), we get
\[\Delta_c^T(\lambda_1,-z)_s^d\pi_{\mathcal{R}}bv_0=(-\zeta)^{dr_j^{(s)}}\pi_{\mathcal{R}}x_{n_{r_l^{(1)},l}\alpha_l}^{\hat{\nu}}(m_{r_l^{(1)},l})\cdots x_{n_{r_1^{(1)},1}\alpha_1}^{\hat{\nu}}(m_{r_1^{(1)},1})\cdots x_{n_{r_1^{(s)}+1,1}\alpha_1}^{\hat{\nu}}(m_{r_1^{(s)}+1,1})\]
\[x_{n_{r_1^{(s)},1}\alpha_1}^{\hat{\nu}}\left(m_{r_1^{(s)},1}+\frac{d}{2}\right)\cdots x_{n_{2,1}\alpha_1}^{\hat{\nu}}\left(m_{2,1}+\frac{d}{2}\right)\left({\bf 1}_T\ot\cdots\ot{\bf 1}_T\ot Y_{\alpha_1}^{\hat{\nu}}\left(-\frac{1}{2}\right){\bf 1}_T\ot\cdots\ot Y_{\alpha_1}^{\hat{\nu}}\left(-\frac{1}{2}\right){\bf 1}_T\right).\]
Since $Y_{\alpha_1}^{\hat{\nu}}\left(-\frac{1}{2}\right){\bf 1}_T=\frac{1}{2}e_{\alpha_1}$, by using (\ref{exe}), we have
\begin{align*}
\zeta^{dr_j^{(s)}}C_{\calR}^{-1}\Delta_c^T(\lambda_1,-z)_s^d\pi_{\mathcal{R}}bv_0=&\left(\frac{1}{2}\right)^sC(\alpha_1,\alpha_2)^{\sum_{p=1}^{r_2^{(1)}}n_{p,2}}C_{\calR^-}^{-1}(1\ot\cdots\ot1\ot e_{\alpha_1}\ot\cdots\ot e_{\alpha_1})\\
\cdot&\pi_{\mathcal{R}^-}x_{n_{r_l^{(1)},l}\alpha_l}^{\hat{\nu}}(m_{r_l^{(1)},l})\cdots x_{n_{1,3}\alpha_3}^{\hat{\nu}}(m_{1,3})x_{n_{r_2^{(1)},2}\alpha_2}^{\hat{\nu}}(m_{r_2^{(1)},2}^\prime)\cdots x_{n_{2,1}\alpha_1}^{\hat{\nu}}(m_{2,1}^\prime)v_0,
\end{align*}
where $\mathcal{R}^-$ is the dual-charge-type given by
\[\mathcal{R}^-=\left(r_l^{(1)},\ldots,r_l^{(k)};\cdots;r_1^{(1)}-1,\ldots,r_1^{(s)}-1,0,\ldots,0\right)\]
and
\begin{align}
\label{m}&m_{p,1}^\prime=m_{p,1}+\frac{d}{2}+n_{p,1}=m_{p,1}-m_{1,1}-\frac{n_{1,1}}{2}+n_{p,1}\quad\text{for }2\leq p\leq r_1^{(s)},\\
\label{m2}&m_{p,1}^\prime=m_{p,1}+n_{p,1}\quad\text{for }r_1^{(s)}<p,\\
\label{m3}&m_{p,2}^\prime=m_{p,2}-\frac{n_{p,2}}{2}.
\end{align}
Note that these $m_{p,i}^\prime$ $(i=1,2)$ satisfy
\begin{enumerate}
\item $m_{p,i}^\prime\leq-\frac{n_{p,i}}{2}+\pair{(\alpha_i)_{(0)},(\alpha_{i-1})_{(0)}}\sum_{q=2}^{r_{i-1}^{(1)}}{\rm min}\{n_{p,i},n_{q,i-1}\}-\sum_{p>p^\prime>0}{\rm min}\{n_{p,i}n_{p^\prime,i}\}$\\
\item $m_{p+1,i}^\prime\leq m_{p,i}^\prime-n_{p,i}\quad\text{if }n_{p+1,i}=n_{p,i}$
\end{enumerate}
for $1\leq p\leq r_i^{(1)}$.
Since the map $e_{\alpha_1}$ is injective, it implies that $\Delta_c^T(\lambda_1,-z)_s^d\pi_{\mathcal{R}}bv_0$ is obtained from the projection of a vector $b^\prime v_0\in\mathcal{B}_W$ which less than $bv_0$ with our linear order.

We can continue to apply this trick for $n_{2,1},\cdots,n_{r_1^{(1)},1}$ and $i=2,\cdots,l$.
Then the image of $bv_0$ is also obtained by the projection of a vector in $\mathcal{B}_W$ which less than that one.

\begin{theorem}\label{basis}
The set $\mathcal{B}_W$ is a basis of $W(k\La_0)$.
\end{theorem}

\begin{proof}
We should prove the linear independence of $\mathcal{B}_W$.
We consider a linear combination of vectors in $\mathcal{B}_W$,
\begin{align}\label{LC}
\sum_{a\in A}c_ab_av_0=0
\end{align}
for a finite set $A$.
Let $b$ be a quasi-particle in (\ref{LC}) with the minimal charge-type $\mathcal{R}^\prime$.
The charge-type $\mathcal{R}^\prime$ determines the dual-charge-type $\mathcal{R}$ and the projection $\pi_{\mathcal{R}}$.
For a quasi-particle $b_{\overline{a}}$ with the charge-type $\overline{\mathcal{R}^\prime}$ greater than $\mathcal{R}^\prime$, we have $\pi_{\mathcal{R}}b_{\overline{a}}v_0=0$ from the definition of the projection.
Therefore we have
\[\sum_{a\in A^\prime}\pi_{\mathcal{R}}c_ab_av_0=0,\]
where $A^\prime$ is a subset of $A$ such that $b_a$ has the charge-type $\mathcal{R}^\prime$ for $a\in A^\prime$.
Namely, all vectors in $\{b_av_0\mid a\in A^\prime\}$ have the same color-charge-type.
We assume that $\overline{b}v_0$ is the smallest vetor in $\{b_av_0\mid a\in A^\prime\}$. 
By applying the above trick to $\overline{b}v_0$, we can reduce the component operators from $\overline{b}v_0$ one by one.
Note that all monomial vectors $b_av_0$ such that $\overline{b}<b_a$ will be annihilated in this step.
Thus we have that the coefficient $\overline{c}$ of $\overline{b}$ is equal to zero.
Continuing the process, we can show that all coefficients are zero.
\end{proof}

\section{bases of standard modules}

\subsection{Spanning sets for standard modules}
To obtain a basis of the standard module, we introduce the following lemma.
Since relations between twisted vertex operators also holds, the proof is parallel to that of Lemma 5 of \cite{OT}.
\begin{lemma}\label{span2}
\phantom{a}
\begin{enumerate}
\item$L(k\La_0)=U(\hat{\heh}[\nu]^-)QW(k\La_0)$\\
\item$L(k\La_0)=QW(k\La_0)$
\end{enumerate}
\end{lemma}

From Lemma \ref{span2}, we know that the set $Q\mathcal{B}_W$ spans $L(k\La_0)$, but it is not its basis.
For example, we recall that the following relation holds
\[x_{\alpha_1}^{\hat{\nu}}\left(-\frac{1}{2}\right)\cdot v_0=\frac{1}{2}e_{\alpha_1}\cdot v_0\]
on $L(\La_0)$.
To get a canonical basis, we consider the following set.
\[B_H=\left\{h_{\alpha_l}\cdots h_{\alpha_1}\left|
\begin{array}{l}
 h_{\alpha_i}=\alpha_i^{\hat{\nu}}(-m_{t_i,i})^{n_{t_i,i}}\cdots\alpha_i^{\hat{\nu}}(-m_{1,i})^{n_{1,i}},i=1,\ldots,l,\\
 t_i\in\Z_{\geq0},m_{t_i,i}>\cdots>m_{1,i},m_{p,i}\in\frac{1}{2}\Z_+,n_{p,i}\in\mathbb{N}
 \end{array}\right.\right\}\]
 We recall that $\alpha^{\hat{\nu}}(m)$ is given by $\alpha_{(4m)}\ot t^m$.
 Since we have
 \[\heh_{(1)}=\heh_{(3)}=\{0\}\]
 for our automorphism $\nu$, we have $\alpha^{\hat{\nu}}(m)=0$ unless $m\in\frac{1}{2}\Z$.
 Then we introduce the linear order on $B_H$.
 A element of $B_H$ has a datum $(n_{t_l,l},\ldots,n_{1,l};\cdots;n_{t_1,1},\ldots,n_{1,1})$ or $(m_{t_l,l},\ldots,m_{1,l};\cdots,;m_{t_1,1},\ldots,m_{1,1})$.
 We can define the order $"<"$ for such datum in the same way as the charge-type $\mathcal{R}^\prime$.
 For two elements $h=h_{\alpha_l}\cdots h_{\alpha_1}$, $\overline{h}=\overline{h}_{\alpha_l}\cdots\overline{h}_{\alpha_1}\in B_H$ of fixed degree, we write $h<\overline{h}$ if
 \begin{enumerate}
 \item$(n_{t_l,l},\ldots,n_{1,1})<(\overline{n}_{\overline{t}_l,l},\ldots,\overline{n}_{1,1})$ or\\
 \item$(n_{t_l,l},\ldots,n_{1,1})=(\overline{n}_{\overline{t}_l,l},\ldots,\overline{n}_{1,1})$ and $(m_{t_l,l},\ldots,m_{1,1})<(\overline{m}_{\overline{t}_l,l},\ldots,\overline{m}_{1,1})$.
 \end{enumerate}
 By combining the order defined in Section 3.1, we generalize this order to the set
 \begin{align}\label{set}
 \{e_\mu hbv_0\mid\mu\in Q,h\in B_H,b\in M_{QP}^\prime\}
 \end{align}
where $M_{QP}^\prime$ is the subset of $M_{QP}$ with no $x_{k\alpha_i}^{\hat{\nu}}(m)$ for $i=1,\ldots,l$.
 For two such vectors $e_\mu hbv_0$, $e_{\bar{\mu}}\bar{h}\bar{b}v_0$ in (\ref{set}) of fixed degree and $\heh_{(0)}$-weight, we denote the color-type of $b$, $\bar{b}$ by $\mathcal{C}$, $\overline{\calC}$ respectively.
 Then we write $e_\mu hbv_0<e_{\bar{\mu}}\bar{h}\bar{b}v_0$ if one of the following conditions holds.
 \begin{enumerate}
 \item${\rm chg}\ b>{\rm chg}\ \bar{b}$,\\
 \item${\rm chg}\ b={\rm chg}\ \bar{b}$ and $\mathcal{C}<\overline{\mathcal{C}}$,\\
 \item$\mathcal{C}=\overline{\mathcal{C}}$ and ${\rm en}\ b<{\rm en}\ \bar{b}$,\\
 \item$\mathcal{C}=\overline{\mathcal{C}}$, ${\rm en}\ b={\rm en}\ \bar{b}$ and $b<\bar{b}$,\\
 \item$b=\bar{b}$ and $h<\bar{h}$.
 \end{enumerate}
Note that $M_{QP}$ is upper bounded with respect to this order.
By induction in our order $"<"$, the following proposition is proved in the same way as Proposition 6 in \cite{OT} (cf. \cite[Lemma 2.3]{BKP}).

\begin{proposition}
The set $\mathcal{B}_L=\{e_\mu hbv_0\mid\mu\in Q,h\in B_H,b\in B_W\cap M_{QP}^\prime\}$ spans $L(k\La_0)$.
\end{proposition}

\subsection{Combinatorial bases of the standard module}
Consider the decomposition
\[L(k\La_0)=\bigoplus_{s\in\Z}L(k\La_0)_s,\quad\text{where }L(k\La_0)_s=\bigoplus_{s_2,\cdots,s_l\in\Z}L(k\La_0)_{s_l\alpha_l+\cdots+s_2\alpha_2+s\alpha_1}.\]
From this decomposition we have the Georgiev-type projection 
\[\pi_{\mathcal{R}_{\alpha_1}}:L(k\La_0)\rightarrow L(\La_0)_{r_1^{(1)}}\ot\cdots\ot L(\La_0)_{r_1^{(k)}}\]
for a fixed dual-charge-type $\mathcal{R}_{\alpha_1}=(r_1^{(1)},r_1^{(2)},\ldots,r_1^{(k)})$ for the color 1 and $r_1=\sum_{s=1}^kr_1^{(s)}$.
We also have the decomposition and the Georgiev-type projection $\pi_{\mathcal{R}_{\alpha_i}}$ for the color $i=2,\ldots,l$.
We will use these projections to prove Theorem \ref{basis2}.
These projections are naturally generalized to
\[L(k\La_0)[[w_{t_l,l}^\frac{1}{2},\ldots,w_{1,1}^\frac{1}{2},z_{r_l^{(1)},l}^\frac{1}{4},\ldots,z_{1,1}^\frac{1}{4}]]\]
and denoted by $\pi_{\mathcal{R}_{\alpha_i}}$.
Set $\alpha^{\hat{\nu}}(w)_-=\sum_{m<0}\alpha^{\hat{\nu}}(m)w^{-m-1}$.
We consider the vector
\[e_\mu\alpha_l^{\hat{\nu}}(-m_{t_l,l}^\prime)^{n_{t_l,l}^\prime}\cdots\alpha_1^{\hat{\nu}}(-m_{1,1}^\prime)^{n_{1,1}^\prime}x_{n_{r_l^{(1)},l}\alpha_l}^{\hat{\nu}}(m_{r_l^{(1)},l})\cdots x_{n_{1,1}\alpha_1}^{\hat{\nu}}(m_{1,1})v_0\]
with dual-charge-type $\mathcal{R}=(\mathcal{R}_{\alpha_l},\ldots,\mathcal{R}_{\alpha_1})$.
We know that the image of this vector with respect to $\pi_{\mathcal{R}_{\alpha_i}}$ is obtained by the coefficient of the corresponding projection of the generating function
\[e_\mu\alpha_l^{\hat{\nu}}(w_{t_l,l})_-^{n_{t_l,l}^\prime}\cdots\alpha_1^{\hat{\nu}}(w_{1,1})_-^{n_{1,1}^\prime}x_{n_{r_l^{(1)},l}\alpha_l}^{\hat{\nu}}(z_{r_l^{(1)},l})\cdots x_{n_{1,1}\alpha_1}^{\hat{\nu}}(z_{1,1})v_0.\]
In \cite{OT}, we defined the generalized twisted vertex operator for an elements of an extension of the weight lattice $P$ of $\geh$ to prove the linear independence of $\mathcal{B}_L$.
In this paper, we continue to use the operator $\Delta_c^T$ to prove the following theorem.

\begin{theorem}\label{basis2}
The set $\mathcal{B}_L$ is a basis of $L(k\La_0)$.
\end{theorem}

\begin{proof}
We prove the linear independence of $\mathcal{B}_L$.
Consider a linear combination of vectors in $\mathcal{B}_L$,
\begin{align}\label{LC2}
\sum c_{\mu,h,b}e_\mu hbv_0=0
\end{align}
of fixed degree and $\heh_{(0)}$-weight.
From (\ref{ehe2}), $e_\mu$ is the bijection such that it maps the weight space $V_\rho$ to $V_{\rho+k\alpha}$ for $\heh_{(0)}$-weight $\rho$. 
Hence we may assume that a summand in (\ref{LC2}) with the maximal charge of color 1, ${\rm chg}_1b=r_1$, has $\mu$ with $\alpha_1$ coordinate zero.
That is, we assume that summands in (\ref{LC2}) has the form
\begin{itemize}
\item[(A)]$e_\mu hbv_0$ with ${\rm chg}_1b=r_1$ and $\mu=c_l\alpha_l+\cdots+c_2\alpha_2$, or
\item[(B)]$e_{\bar{\mu}}\bar{h}\bar{b}v_0$ with ${\rm chg}_1\bar{b}<r_1$ and $\bar{\mu}=\bar{c}_l\alpha_l+\cdots+\bar{c}_1\alpha_1$, where $\bar{c}_1>0$.
\end{itemize}
Among the vectors $v=e_\mu hbv_0$ with ${\rm chg}_1b=r_1$, we choose a vector with the maximal charge-type $\mathcal{R}_{\alpha_1}^\prime$ and corresponding dual-charge-type
\[\mathcal{R}_{\alpha_1}=\left(r_1^{(1)},\ldots,r_1^{(k-1)},0\right)\]
for the color $i=1$, where $r_1=r_1^{(1)}+\cdots+r_1^{(k-1)}$.
Note that $r_1^{(k)}=0$ for a vector in $\mathcal{B}_L$.
Since the action of $e_\alpha$ on $L(k\La_0)$ is given by
\[e_\alpha\cdot v_0=e_\alpha{\bf 1}_T\ot\cdots\ot e_\alpha{\bf 1}_T,\]
we have
\[e_{\bar{\mu}}\bar{h}\bar{b}v_0\in\bigoplus_{\substack{s_1,\cdots,s_{k-1}\in\Z\\
s_k>0}}L(\La_0)_{s_1}\ot\cdots\ot L(\La_0)_{s_k}.\]
Therefore we have $\pi_{\mathcal{R}_{\alpha_1}}(e_{\bar{\mu}}\bar{h}\bar{b})v_0=0$ for the vector of the form (B).
After applying the projection $\pi_{\mathcal{R}_{\alpha_1}}$ to the sum (\ref{LC2}), we have only the summands of the form (A).
From (\ref{Delta2}), (\ref{Delta4}), we can apply $\Delta_c^T(\lambda_1,-z)$ to $\pi_{\mathcal{R}_{\alpha_1}}(e_\mu hb)v_0$ and it affect only $b$.
Now we choose the smallest  monomial $b$ in the summands of the form (A).
Using the same way to the proof of Theorem \ref{basis}, we can reduce the color 1 quasi-particles from $e_\mu hbv_0$ one by one.
Then we have the vector $e_\mu hb^\prime v_0$ with ${\rm chg}_1b^\prime=0$.
By applying the same trick for the color $i=2,\ldots,l$, we have $c_{\mu,h,b}=0$.
Continuing this process, we can show the linear independence of $\mathcal{B}_L$.
\end{proof}

\section{parafermionic bases}

In this section we define the parafermionic space as in \cite{BKP,OT}.
By the same argument in \cite{OT}, we obtain the parafermionic bases of the parafermionic space.

\subsection{Vacuum space and twisted $\mathcal{Z}$-operators}

We denote the vacuum space of the standard module $L(k\La_0)$ by $L(k\La_0)^{\hat{\heh}[\nu]^+}$.
That is we have
\begin{align}\label{vac sp}
L(k\La_0)^{\hat{\heh}[\nu]^+}=\{v\in L(k\La_0)\mid \hat{\heh}[\nu]^+\cdot v=0\}.
\end{align}
From the Lepowsky-Wilson theorem \cite{LW} (A5.3) we have canonical isomorphism of $d$-graded linear spaces
\begin{align}\label{LWthm}
U(\hat{\heh}[\nu]^-)\ot L(k\La_0)^{\hat{\heh}[\nu]^+}&\overset{\simeq}{\longrightarrow} L(k\La_0)\\
\nonumber h\ot u&\mapsto h\cdot u
\end{align}
where $U(\hat{\heh}[\nu]^-)\simeq{\rm Sym}(\hat{\heh}[\nu]^-)$ is the Fock space of level $k$ for the Heisenberg subalgebra $\hat{\heh}[\nu]_{\frac{1}{4}\Z}$ with the action of $c$ being the multiplication by scalar $k$.
This isomorphism gives the direct decomposition
\begin{align}\label{decomposition}
L(k\La_0)=L(k\La_0)^{\hat{\heh}[\nu]^+}\oplus \hat{\heh}[\nu]^-U(\hat{\heh}[\nu]^-)L(k\La_0)^{\hat{\heh}[\nu]^+}.
\end{align}
Then we obtain the projection
\[\pi^{\hat{\heh}[\nu]^+}:L(k\La_0)\rightarrow L(k\La_0)^{\hat{\heh}[\nu]^+}\]
from (\ref{decomposition}).
We define the $\mathcal{Z}$-operator by
\[\mathcal{Z}_{n\alpha}(z)=E^-(\alpha,z)^{n/k}x_{n\alpha}^{\hat{\nu}}(z)E^+(\alpha,z)^{n/k}\]
for a quasi-particle of charge $n$ and  a root $\alpha$.
Note that the action of $\mathcal{Z}$-operators commute with the action of the Heisenberg subalgebra $\hat{\heh}[\nu]_{\frac{1}{4}\Z}$ on the standard module $L(k\La_0)$.
More generally, we should define the $\mathcal{Z}$-operators for quasi-particles of charge-type $\mathcal{R}^\prime=(n_{r_l^{(1)},l},\ldots,n_{1,1})$.
For the twisted vertex operator $x_{\mathcal{R}^\prime}^{\hat{\nu}}(z_{r_l^{(1)},l},\ldots,z_{1,1})=x_{n_{r_l^{(1)},l}\alpha_l}^{\hat{\nu}}(z_{r_l^{(1)},l})\cdots x_{n_{1,1}\alpha_1}^{\hat{\nu}}(z_{1,1})$ of charge-type $\mathcal{R}^\prime$, we define
\begin{align}\label{Zop}
\nonumber\mathcal{Z}_{\mathcal{R}^\prime}(z_{r_l^{(1)},l},\ldots,z_{1,1})=&E^-(\alpha_l,z_{r_l^{(1)},l})^{n_{r_l^{(1)},l}/k}\cdots E^-(\alpha_1,z_{1,1})^{n_{1,1}/k}x_{\mathcal{R}^\prime}^{\hat{\nu}}(z_{r_l^{(1)},l},\ldots,z_{1,1})\\
&\times E^+(\alpha_l,z_{r_l^{(1)},l})^{n_{r_l^{(1)},l}/k}\cdots E^+(\alpha_1,z_{1,1})^{n_{1,1}/k}.
\end{align}
For convenience, we write this formal Laurent series by
\[\mathcal{Z}_{\mathcal{R}^\prime}(z_{r_l^{(1)},l},\ldots,z_{1,1})=\sum_{m_{r_l^{(1)},l},\ldots,m_{1,1}\in\frac{1}{4}\Z}\mathcal{Z}_{\mathcal{R}^\prime}(m_{r_l^{(1)},l},\ldots,m_{1,1})z_{r_l^{(1)},l}^{-m_{r_l^{(1)},l}-n_{r_l^{(1)},l}}\cdots z_{1,1}^{-m_{1,1}-n_{1,1}}.\]
Since $\mathcal{Z}$-operators are well-defined on vacuum space and we can express quasi-particle monomials in terms of $\mathcal{Z}$-operators by reversing (\ref{Zop}), we have
\[\pi^{\hat{\heh}[\nu]^+}x_{\mathcal{R}^\prime}^{\hat{\nu}}(z_{r_l^{(1)},l}\ldots,z_{1,1})v_0\mapsto\mathcal{Z}_{\mathcal{R}^\prime}(z_{r_l^{(1)},l},\ldots,z_{1,1})v_0.\]

Now, Theorem \ref{basis2} implies
\begin{theorem}\label{basis3}
The set of vectors
\[e_\mu\pi^{\hat{\heh}[\nu]^+}(b)v_0=e_\mu\mathcal{Z}_{\mathcal{R}^\prime}(m_{r_l^{(1)},l},\ldots,m_{1,1})v_0\]
such that $\mu\in Q$ and 
$b\in B_W\cap M_{QP}^\prime$ with the charge-type $\calR^\prime$ and the energy-type $(m_{r_l^{(1)},l},\ldots,m_{1,1})$ is a basis of the vacuum space $L(k\La_0)^{\hat{\heh}[\nu]^+}$.
\end{theorem}

The proof is parallel to that of Theorem 3.1 of \cite{BKP}.

\subsection{Parafermionic space and parafermionic current}
We have the projective representation of $Q$ on $L(k\La_0)$.
This gives a diagonal action of the sublattice $kQ\subset Q$ by
\[k\alpha\mapsto\rho(k\alpha)=e_\alpha\ot\cdots\ot e_\alpha.\]
Note that this action satisfies $\rho(k\alpha):L(k\La_0)_\mu^{\hat{\heh}[\nu]^+}\rightarrow L(k\La_0)_{\mu+k\alpha}^{\hat{\heh}[\nu]^+}$ for a weight $\mu$.
We define the parafermionic space of the highest weight $k\La_0$ as the space of $kQ$-coinvariants in the $kQ$-module $L(k\La_0)^{\hat{\heh}[\nu]^+}$:
\begin{align}\label{para sp}
L(k\La_0)_{kQ}^{\hat{\heh}[\nu]^+}:=L(k\La_0)^{\hat{\heh}[\nu]^+}/{\rm Span}_{\C}\{(\rho(k\alpha)-1)\cdot v\mid\alpha\in Q,v\in L(k\La_0)^{\hat{\heh}[\nu]^+}\}.
\end{align}
We have the canonical projection
\[\pi_{kQ}^{\hat{\heh}[\nu]^+}:L(k\La_0)^{\hat{\heh}[\nu]^+}\rightarrow L(k\La_0)_{kQ}^{\hat{\heh}[\nu]^+}.\]
The composition $\pi_{kQ}^{\hat{\heh}[\nu]^+}\circ\pi^{\hat{\heh}[\nu]^+}$ is denoted by $\pi$.
Then we have
\[L(k\La_0)_{kQ}^{\hat{\heh}[\nu]^+}\simeq\bigoplus_{\mu\in k\La_0+Q/kQ}L(k\La_0)_\mu^{\hat{\heh}[\nu]^+}.\]

For every root $\beta$, we define the parafermionic current of charge $n$ by
\begin{align}\label{para c}
\Psi_{n\beta}^{\hat{\nu}}(z)=\mathcal{Z}_{n\beta}(z)z^{-n\beta_{(0)}/k}\ep_\beta^{-n/k},
\end{align}
where $\ep_\beta:L(k\La_0)\rightarrow\C^\times$ is given by
\[\ep_\beta u=C(\beta,\mu)u\quad\text{for }u\in L(k\La_0)_\mu.\]
The commutativity with the action of Heisenberg subalgebra of $\mathcal{Z}$-operator implies that the parafermionic current preserves the vacuum space $L(k\La_0)^{\hat{\heh}[\nu]^+}$.
We rewrite the commutation relation (\ref{exe2}) by
\[x_\beta^{\hat{\nu}}(z)e_\alpha=C(\alpha,\beta)^{-1}e_\alpha x_\beta^{\hat{\nu}}(z)z^{\pair{\alpha,\beta_{(0)}}}.\]
Using this relation and the one between $z^\mu$ and $e_\alpha$, we have
\[[\rho(k\alpha),\Psi_\beta^{\hat{\nu}}]=0,\]
where the map $\ep_\beta$ contribute to vanishing constant term $C(\alpha,\beta)$.
Therefore $\Psi_\beta^{\hat{\nu}}(z)$ is well-defined on the parafermionic space $L(k\La_0)_{kQ}^{\hat{\heh}[\nu]^+}$.
For a quasi-particle of charge-type $\mathcal{R}^\prime=(n_{r_l^{(1)},l},\ldots,n_{1,1})$, the parafermionic current of charge-type $\mathcal{R}^\prime$ is also defined by
\[\Psi_{\mathcal{R}^\prime}^{\hat{\nu}}(z_{r_l^{(1)},l},\ldots,z_{1,1})=\mathcal{Z}_{\mathcal{R}^\prime}(z_{r_l^{(1)},l},\ldots,z_{1,1})z_{r_l^{(1)},l}^{-n_{r_l^{(1)},l}(\alpha_l)_{(0)}/k}\cdots z_{1,1}^{-n_{1,1}(\alpha_1)_{(0)}/k}\ep_{\alpha_l}^{-n_{r_l^{(1)},l}/k}\cdots\ep_{\alpha_1}^{-n_{1,1}/k}.\]
Note that the parafermionic current of charge-type $\mathcal{R}^\prime$ also commute with the diagonal action $\rho(k\alpha)$ for $\alpha\in Q$.
As in the $\mathcal{Z}$-operator, we set
\[\Psi_{\mathcal{R}^\prime}^{\hat{\nu}}(z_{r_l^{(1)},l},\ldots,z_{1,1})=\sum_{m_{r_l^{(1)},l},\ldots,m_{1,1}}\psi_{\mathcal{R}^\prime}^{\hat{\nu}}(m_{r_l^{(1)},l},\ldots,m_{1,1})z_{r_l^{(1)},l}^{-m_{r_l^{(1)},l}-n_{r_l^{(1)},l}}\cdots z_{1,1}^{-m_{1,1}-n_{1,1}}\]
where the summation is over all sequences $(m_{r_l^{(1)},l},\ldots,m_{1,1})$ such that $m_{p,i}\in\frac{1}{4}\Z+\frac{n_{p,i}\pair{(\alpha_i)_{(0)},\mu}}{k}$ on the $\mu$-weight space $L(k\La_0)^{\hat{\heh}[\nu]^+}$.

Then we introduce several lemmas which give the parafermionic bases of the space $L(k\La_0)_{kQ}^{\hat{\heh}[\nu]^+}$.
First, next lemma associates the coefficients of $\mathcal{Z}$-operators with those of parafermionic currents (cf. \cite[Lemma 10]{OT}).
\begin{lemma}\label{Zpara}
For a simple root $\beta$, $m\in\frac{1}{4}\Z$ and weight $\mu$ we have
\[\left.\mathcal{Z}_\beta(m)\mathrel{}\middle|\mathrel{}_{L(k\La_0)_\mu^{\hat{\heh}[\nu]^+}}=C(\beta,\mu)\psi_\beta^{\hat{\nu}}(m+\pair{\beta_{(0)},\mu}/k)\mathrel{}\middle|\mathrel{}_{L(k\La_0)_\mu^{\hat{\heh}[\nu]^+}}\right..\]
\end{lemma}
Next, we consider the relation between different parafermionic currents.
The following lemmas are obtained by direct computation.
The proof are parallel to that of Lemma 3.2 and 3.3 of \cite{BKP} respectively.
\begin{lemma}\label{para c2}
For a simple root $\beta$ and a positive integer $n$,
\begin{align}\label{para de}
\Psi_{n\beta}^{\hat{\nu}}(z)=\left.\left(\prod_{1\leq p<s\leq n}\prod_{i=0}^{3}(z_s^\frac{1}{4}-\zeta^iz_p^\frac{1}{4})^{\pair{\nu^i\beta,\beta}/k}\right)\Psi_\beta^{\hat{\nu}}(z_n)\cdots\Psi_\beta^{\hat{\nu}}(z_1)\right|_{z_n=\cdots=z_1=z}.
\end{align}
\end{lemma}
\noindent Note that we use the fact that $C(\beta,\beta)=1$ for a simple root in the proof of the lemma \ref{para c2}.
For simplicity, we set
\[\Psi_{n_t\beta_t,\ldots,n_1\beta_1}(z_t,\ldots,z_1)=\mathcal{Z}_{(n_t,\ldots,n_1)}(z_t,\ldots,z_1)\prod_{i=1}^tz_i^{-n_i(\beta_i)_{(0)}/k}\ep_{\beta_i}^{-n_i/k}\]
for a given simple roots $\beta_t,\ldots,\beta_1$ and charges $n_t,\ldots,n_1$.
\begin{lemma}\label{para c3}
\begin{align}\label{para de2}
\nonumber\Psi_{n_t\beta_t,\ldots,n_1\beta_1}^{\hat{\nu}}&(z_t,\ldots,z_1)\\
&=\left(\prod_{1\leq p<s\leq t}C(\beta_s,\beta_p)^{n_sn_p/k}\prod_{i=0}^3(z_s^\frac{1}{4}-\zeta^iz_p^\frac{1}{4})^{\pair{n_s\nu^i\beta_s,n_p\beta_p}/k}\right)\Psi_{n_t\beta_t}^{\hat{\nu}}(z_t)\cdots\Psi_{n_1\beta_1}^{\hat{\nu}}(z_1).
\end{align}
\end{lemma}
From Theorem \ref{basis3}, we have
\begin{theorem}\label{basis4}
For the highest weight $k\La_0$, the set of vectors
\[\pi_{kQ}^{\hat{\heh}[\nu]^+}\mathcal{Z}_{\mathcal{R}^\prime}(m_{r_l^{(1)},l},\ldots,m_{1,1})v_0=\psi_{\mathcal{R}^\prime}(m_{r_l^{(1)},l},\ldots,m_{1,1})v_0\]
is a basis of the paragermionic space $L(k\La_0)_{kQ}^{\hat{\heh}[\nu]^+}$,  where $\mathcal{Z}_{\mathcal{R}^\prime}(m_{r_l^{(1)},l},\ldots,m_{1,1})v_0$ is a vector that appears in the basis of the vacuum space $L(k\La_0)^{\hat{\heh}[\nu]^+}$.
\end{theorem}

\section{The fermionic character formula}

By using the quasi-particle bases and the parafermionic bases, we can calculate the character of the principal subspace and the parafermionic space.
Furthermore, combining the main theorem and Lepowsky-Wilson theorem (\ref{LWthm}), we will obtain the fermionic character formula of the standard module.

\subsection{Principal subspace}

On the principal subspace $W(k\La_0)$, we use the weight gradation by $L^{\hat{\nu}}(0)$ or equivalent to $d$.
That is, we have
\[[L^{\hat{\nu}}(0),x_{\alpha}^{\hat{\nu}}(m)]=\left(-m-1+\frac{\pair{\alpha,\alpha}}{2}\right)x_\alpha^{\hat{\nu}}(m)=-mx_\alpha^{\hat{\nu}}(m).\]

We define the character of the principal subspace $W(k\La_0)$ by
\[{\rm ch}\ W(k\La_0)=\sum_{m,r_1,\ldots,r_l\geq0}{\rm dim}\ W(k\La_0)_{(m,r_1,\ldots,r_l)}q^my_1^{r_1}\cdots y_l^{r_l},\]
where $W(k\La_0)_{(m,r_1,\ldots,r_l)}$ is the weight subspace spanned by monomial vectors of the weight $k\La_0-m\delta+r_1\alpha_1+\cdots+r_l\alpha_l$ with respect to $\heh_{(0)}\op\C d$.

For an arbitrary quasi-particle monomial of the form (\ref{QP}), we define the sequence  $\mathcal{P}_i=(p_i^{(1)},\ldots,p_i^{(k)})$ by $p_i^{(s)}=r_i^{(s)}-r_i^{(s+1)}$ for $i=1,\ldots,l$, $s=1,\ldots,k$ so that $p_i^{(s)}$ stand for the number of quasi-particles of color $i$ and charge $s$ in the monomial (\ref{QP}).
Set $\mathcal{P}=(\mathcal{P}_l,\ldots,\mathcal{P}_1)$.
We then rewrite the condition (C2) on the energies in terms of $\mathcal{P}$.
For a fixed charge-type $\mathcal{R}^\prime=(n_{r_l^{(1)},l},\ldots,n_{1,1})$, we have
\[\sum_{p=1}^{r_i^{(1)}}\rho_i(2p-1)n_{p,i}=\rho_i\sum_{s,t=1}^k{\rm min}\{s,t\}p_i^{(s)}p_i^{(t)},\]
\[\frac{1}{2}\sum_{p=1}^{r_i^{(1)}}\sum_{q=1}^{r_{i-1}^{(1)}}{\rm min}\{n_{p,i},n_{q,i-1}\}=\frac{1}{2}\sum_{s,t=1}^k{\rm min}\{s,t\}p_i^{(s)}p_{i-1}^{(t)}.\]
These expression are proved by induction on the level of the standard module (cf. \cite{G}).
We write
\begin{align*}
(q^\frac{1}{2})_n=\begin{cases}
\prod_{i=1}^n(1-q^\frac{1}{2})&(n>0)\\
1&(n=0)
\end{cases}\quad(q^\frac{1}{2})_\infty=\prod_{i\geq1}(1-q^\frac{1}{2}).
\end{align*}
From \cite{A}, we have
\[\frac{1}{(q^\frac{1}{2})_n}=\sum_{j\geq0}p_n(j)q^\frac{j}{2},\]
where $p_n(j)$ is the number of partitions of $j$ with most $n$ parts.
Now, we are able to calculate the character of the principal subspace $W(k\La_0)$ as follows.
\begin{theorem}\label{character}
For affine Lie algebras $A_{2l}^{(2)}$, we have
\[{\rm ch}\ W(k\La_0)=\sum_{\mathcal{P}}\frac{q^{\frac{1}{2}\sum_{i,j=1}^l\sum_{s,t=1}^k\pair{(\alpha_i)_{(0)},(\alpha_j)_{(0)}}{\rm min}\{s,t\}p_i^{(s)}p_j^{(t)}}}{\prod_{i=1}^l\prod_{s=1}^k(q^\frac{1}{2})_{p_i^{(s)}}}\prod_{i=1}^ly_i^{\sum_{s=1}^ksp_i^{(s)}}\]
where the sum goes over all finite sequences $\mathcal{P}$ of $lk$ nonnegative integers.
\end{theorem}

\subsection{Parafermionic space}
Using the corresponding result of the principal subspace, we calculate the character of the parafermionic space.
For untwisted Lie algebras, we use the parafermionic grading operator defined by (3.35) in \cite{Li2}.
We should modify the grading operator $L^{\hat{\nu}}(0)$ as well.
But we do not find the coset Virasoro algebra construction \cite[\S 3]{Li2} for $\hat{\nu}$-twisted $V_L$-module.
Therefore we replace the grading operator $L^{\hat{\nu}}(0)$ by $D$ defined by
\[D=-d-D^{\hat{\heh}[\nu]^+},\quad \left.D^{\hat{\heh}[\nu]^+}\right|_{L(k\La_0)_\mu^{\hat{\heh}[\nu]^+}}=\frac{\pair{\mu_{(0)},\mu_{(0)}}}{2k}\]
for a weight $\mu$ (cf. \cite{OT}).
For a simple root $\beta\in L$ and $m\in\frac{1}{4}\Z$ , we have
\[[D,x_\beta^{\hat{\nu}}(m)]=\left(-m-\frac{\pair{\beta_{(0)},\beta_{(0)}}}{2k}\right)x_\beta^{\hat{\nu}}(m).\]
Then we also have
\[[D,\psi_\beta^{\hat{\nu}}(m)]=\left(-m-\frac{\pair{\beta_{(0)},\beta_{(0)}}}{2k}\right)\psi_\beta^{\hat{\nu}}(m).\]
The conformal energy of $\psi_\beta^{\hat{\nu}}(m)$ is defined as the coefficient of the right hand side and denoted by
\[{\rm en}\ \psi_\beta^{\hat{\nu}}(m)=-m-\frac{\pair{\beta_{(0)},\beta_{(0)}}}{2k}.\]
Now we can compute the conformal energies of $\psi_{n\beta}^{\hat{\nu}}(m)$ and $\psi_{n_t\beta_t,\ldots,n_1\beta_1}^{\hat{\nu}}(m_t,\ldots,m_1)$ in the same way as \cite[Lemma 14]{OT}.
\begin{lemma}\label{para en}
For a simple root $\beta$ and charge $n$, we have
\[{\rm en}\ \psi_{n\beta}^{\hat{\nu}}(m)=-m-\frac{n^2\pair{\beta_{(0)},\beta_{(0)}}}{2k}.\]
Moreover, for simple roots $\beta_t,\ldots,\beta_1$ and charges $n_t,\ldots,n_1$, we have
\[{\rm en}\ \psi_{n_t\beta_t,\ldots,n_1\beta_1}^{\hat{\nu}}(m_t,\ldots,m_1)=\sum_{i=1}^t\left({\rm en}\ \psi_{n_i\beta_i}^{\hat{\nu}}(m_i)-\sum_{p=1}^{i-1}\frac{\pair{n_i(\beta_i)_{(0)},n_p(\beta_p)_{(0)}}}{k}\right).\]
\end{lemma}
Using (\ref{ede2}), we find that $[D,\rho(k\alpha)]=0$ for $\alpha\in Q$.
Thus the grading operator $D$ is well-defined on the parafermionic space $L(k\La_0)_{kQ}^{\hat{\heh}[\nu]^+}$.

We define the character of the parafermionic space $L(k\La_0)_{kQ}^{\hat{\heh}[\nu]^+}$ by
\[{\rm ch}\ L(k\La_0)_{kQ}^{\hat{\heh}[\nu]^+}=\sum_{m,r_1,\ldots,r_l\geq0}{\rm dim}(L(k\La_0)_{kQ}^{\hat{\heh}[\nu]^+})_{(m,r_1,\ldots,r_l)}q^my_1^{r_1}\cdots y_l^{r_l},\]
where $(L(k\La_0)_{kQ}^{\hat{\heh}[\nu]^+})_{(m,r_1,\ldots,r_l)}$ is the weight subspace spanned by monomial vectors of conformal energy $m$ and color-type $\mathcal{C}=(r_l,\ldots,r_1)$.

Consider the quasi-particle monomial
\begin{align}\label{QP mono}
x_{n_{r_l^{(1)},l}\alpha_l}^{\hat{\nu}}(m_{r_l^{(1)},l})\cdots x_{n_{1,l}\alpha_l}^{\hat{\nu}}(m_{1,l})\cdots x_{n_{r_1^{(1)},1}\alpha_1}^{\hat{\nu}}(m_{r_1^{(1)},1})\cdots x_{n_{1,1}\alpha_1}^{\hat{\nu}}(m_{1,1})\in B_W\cap M_{QP}^\prime
\end{align}
with charge-type $\mathcal{R}^\prime=(n_{r_l^{(1)},l},\ldots,n_{1,1})$, dual-charge-type $\mathcal{R}=(r_l^{(1)},\ldots,r_1^{(k-1)},0)$.
Note that $p_i^{(k)}$ (or equivalently $r_i^{(k)}$) is equal to zero for all $i=1,\ldots,l$, in (\ref{QP mono}).
To emphasize the dependence of $k$, we replace $\mathcal{P}$ by $\mathcal{P}^{(k-1)}$.
We consider the parafermionic basis vector
\[\psi_{n_{r_l^{(1)},l}\alpha_l,\ldots,n_{1,1}\alpha_1}^{\hat{\nu}}(m_{r_l^{(1)},l},\ldots,m_{1,1})\]
which correspond to the monomial (\ref{QP mono}).
From lemma \ref{para en}, the conformal energy of this current is given by
\begin{align}
\nonumber&-\sum_{i=1}^l\left(\sum_{s=1}^{r_i^{(1)}}m_{s,i}+\sum_{s=1}^{r_i^{(1)}}\frac{n_{s,i}^2\rho_i}{k}+\sum_{s=1}^{r_i^{(1)}}\left(\sum_{t=1}^{s-1}\frac{2n_{s,i}n_{t,i}\rho_i}{k}+\sum_{j=1}^{i-1}\sum_{t=1}^{r_k^{(1)}}\frac{\pair{n_{s,i}(\alpha_i)_{(0)},n_{t,j}(\alpha_j)_{(0)}}}{k}\right)\right)\\
\nonumber&=-\sum_{i=1}^l\sum_{s=1}^{r_i^{(1)}}m_{s,i}-\frac{1}{2}\sum_{i,j=1}^l\sum_{s=1}^{r_i^{(1)}}\sum_{t=1}^{r_j^{(1)}}\frac{\pair{n_{s,i}(\alpha_i)_{(0)},n_{t,j}(\alpha_j)_{(0)}}}{k}\\
&=-\sum_{i=1}^l\sum_{s=1}^{r_i^{(1)}}m_{s,i}-\frac{1}{2}\sum_{i,j=1}^l\sum_{s,t=1}^{k-1}\frac{st}{k}\pair{(\alpha_i)_{(0)},(\alpha_j)_{(0)}}p_i^{(s)}p_j^{(t)},
\end{align}
where we use the fact that
\[\sum_{s=1}^{r_i^{(1)}}n_{s,i}=\sum_{s=1}^{k-1}sp_i^{(s)}.\]
Combining Theorem \ref{character} and the contribution of the conformal shift, we obtain the character of the parafermionic space $L(k\La_0)_{kQ}^{\hat{\heh}[\nu]^+}$.
\begin{theorem}\label{character2}
For affine Lie algebras $A_{2l}^{(2)}$, we obtain
\[{\rm ch}\ L(k\La_0)_{kQ}^{\hat{\heh}[\nu]^+}=\sum_{\mathcal{P}^{(k-1)}}\frac{q^{\frac{1}{2}\sum_{i,j=1}^l\pair{(\alpha_i)_{(0)},(\alpha_j)_{(0)}}\sum_{s,t=1}^{k-1}D_{s,t}^{(k)}p_i^{(s)}p_j^{(t)}}}{\prod_{i=1}^l\prod_{s=1}^{k-1}(q^\frac{1}{2})_{p_i^{(s)}}}\prod_{i=1}^ly_i^{\sum_{s=1}^{k-1}sp_i^{(s)}}\]
where the sum runs over all sequences $\mathcal{P}^{(k-1}$ of $l(k-1)$ nonnegative integers and
\[D_{s,t}^{(k)}={\rm min}\{s,t\}-\frac{st}{k}.\]
\end{theorem}

\subsection{Standard module}
Finally, we calculate the character of the standard module $L(k\La_0)$.
The character of the standard module is defined in  the same way as that of principal subspace.
We see that the character formula of the standard module is given as follows.
\begin{theorem}\label{ST ch}
For affine Lie algebras $A_{2l}^{(2)}$, we have
\begin{align*}
{\rm ch}\ L(k\La_0)=\frac{1}{\prod_{i=1}^l(q^\frac{1}{2})_\infty}\sum_{\eta\in Q_{(0)}}q^{\pair{\eta,\eta}/2k}\prod_{i=1}^ly_i^{\eta_i}\sum_{\calP^{(k-1)}}\frac{q^{\frac{1}{2}\sum_{i,j=1}^l\pair{(\alpha_i)_{(0)},(\alpha_j)_{(0)}}\sum_{s,t=1}^{k-1}D_{s,t}^{(k)}p_i^{(s)}p_j^{(t)}}}{\prod_{i=1}^l\prod_{s=1}^{k-1}(q^\frac{1}{2})_{p_i^{(s)}}}
\end{align*}
where $\eta_i\in\Z(\alpha_i)_{(0)}$ and the sum $\sum_{\calP^{(k-1)}}$ runs over all sequences $\calP^{(k-1)}$ of $l(k-1)$ nonnegative integers satisfying
\[\sum_{i=1}^l\sum_{s=1}^{k-1}sp_i^{(s)}(\alpha_i)_{(0)}\in\eta+kQ_{(0)}.\]
\end{theorem}

In fact, using Lepowsky-Wilson theorem (\ref{LWthm}), we have the following relation for the character formula
\[{\rm ch}\ L(k\La_0)=\frac{1}{\prod_{i=1}^l(q^\frac{1}{2})_\infty}{\rm ch}\ L(k\La_0)^{\hat{\heh}[\nu]^+}.\]
Since the basis of the vacuum space $L(k\La_0)^{\hat{\heh}[\nu]^+}$ is given in Theorem \ref{basis3}, we are able to calculate ${\rm ch}\ L(k\La_0)^{\hat{\heh}[\nu]^+}$ in the same way as \cite[Theorem 17]{OT}.

\section*{Acknowledgments}
The author would like to thank Masato Okado for helpful comments and discussion on this research.
This work is supported by JST, the establishment of university fellowships towards the creation of science technology innovation, Grant Number JPMJFS2138.


\begin{thebibliography}{99}
\bibitem{A}G. E. Andrews, The theory of partitions. Addison-Wesley 1976.
\bibitem{B}M. Butorac, A note on principal subspaces of the affine Lie algebras in types $B_l^{(1)},C_l^{(1)},F_4^{(1)}\text{ and }G_2^{(1)}$, Comm. Algebra {\bf 48}, 5343-5359 (2020).
\bibitem{BK}M. Butorac and S. Ko\v{z}i\'{c}, Principal subspaces for the affine Lie algebras in type $D,E\text{ and }F$, preprint arXiv:1902.10794.
\bibitem{BKP}M. Butorac, S. Ko\v{z}i\'{c} and M. Primc, Parafermionic bases of standard modules for affine Lie algebras, Mathematische Zeitschrift (2020), published online, https://doi.org/10.1007/s00209-020-02639-w.
\bibitem{BS}M. Butorac and C. Sadowski, Combinatorial bases of principal subspace of modules for twisted affine Lie algebras of type $A_{2l-1}^{(2)}$, $D_l^{(2)}$, $E_6^{(2)}$ and $D_4^{(3)}$, New York J. Math. {\bf 25} (2019), 79-106.
\bibitem{CLM}C. Calinescu, J. Lepowsky, A. Milas, Vertex-algebraic structure of principal subspaces of standard $A_2^{(2)}$-modules, I. Internat. J. Math. {\bf 25} (2014), 1450063.
\bibitem{CMP}C. Calinescu, A. Milas, M. Penn, Vertex algebraic structure of principal subspaces of basic $A_{2n}^{(2)}$-modules, J. Pure Appl. Algebra {\bf 220} (2016), 1752-1784.
\bibitem{DL}C. Dong and J. Lepowsky, The algebraic structure of relative twisted vertex operators, J. Pure Appl. Algebra {\bf 110} (1996) 259-295.
\bibitem{FS}B. Feigin and A. Stoyanovsky, Quasi-particles models for the representations of Lie algebras and geometry of flag manifold, arXiv:hep-th/9308079.
\bibitem{FS2}B. Feigin and A. Stoyanovsky, Functional models for representations of current algebras and semi-infinite Schubert cells (Russian), Funksional Anal i Prilozhen, {\bf 28} (1994), 68-90; translation in: Funct. Anal. Appl., {\bf 28} (1994), 55-72.
\bibitem{FLM}I. Frenkel, J. Lepowsky and A. Meurman, Vertex operator calculus, in Mathematical Aspects of String Theory, ed. S.-T. Yau (World Scientific, Singapore, 1987), pp. 150-188.
\bibitem{FLM2}I. Frenkel, J. Lepowsky and A. Meurman, Vertex Operator Algebras and the Monster, Pure and Applied Mathematics, Vol. 134 (Academic Press, 1988).
\bibitem{G}G. Georgiev, Combinatorial construction of modules for infinite dimensional Lie algebras, I. Principal subspace, J. Pure Appl. Algebra {\bf 112} (1996), 247-286.
\bibitem{G2}G. Georgiev, Combinatorial construction of modules for infinite-dimensional Lie algebras, II, Parafermionic space, arXiv:q-alg/9504024.
\bibitem{HKOTT}G. Hatayama, A. Kuniba, M. Okado, T. Takagi and T. Tsuboi, Path, Crystals and Fermionic Formulae, MathPhys Odyssey 2001, 205-272, Prog. Math. Phys. {\bf 23}, Birkh\"{a}user Boston, MA, 2002. 
\bibitem{K}V.G. Kac, Infinite dimensional Lie algebras, 3rd ed., Cambridge University Press, Cambridge, 1990.
\bibitem{KNS}A. Kuniba, T. Nakanishi and J. Suzuki, Characters in conformal field theories from thermodynamic Bethe Ansatz, Modern Phys. Lett. {\bf A8} (1993), 1649-1659.
\bibitem{L}J. Lepowsky, Calculus of twisted vertex operators, Proc. Nat. Acid. Sci. USA, {\bf 82} (1985), 8295-8299.
\bibitem{LP}J. Lepowsky and M. Primc, Structure of the standard modules for the affine Lie algebra $A_1^{(1)}$, Contemporary Math., {\bf 46}, Amer. Math. Soc, Providence, RI, 1985.
\bibitem{LW2}J. Lepowsky and R.L. Wilson, Construction of the affine Lie algebra $A_1^{(1)}$, Comm. Math. Phys. {\bf 62} (1978), 43-53.
\bibitem{LW}J. Lepowsly and R.L. Wilson, The structure of standard modules, I: Universal algebras and the Rogers-Ramanujan identities, Invent. Math. {\bf 77} (1984), 199-290.
\bibitem{Li}H.-S. Li, The physics superselection principle in vertex operator algebra theory, J. Algebra {\bf 196} (1997) 436-457.
\bibitem{Li2}H.-S. Li, On abelian coset generalized vertex algebras, Commun, Contemp. Math. {\bf 03}, No. 02, (2001), 287-340.
\bibitem{OT}M. Okado and R. Takenaka, Parafermionic bases of standard modules for twisted affine Lie algebras of type $A_{2l-1}^{(2)}$, $D_{l+1}^{(2)}$, $E_6^{(2)}$ and $D_4^{(3)}$, arXiv:2109.08892v1.
\bibitem{P}M. Primc, Vertex operator construction of standard modules for $A_n^{(1)}$, Pacific J. Math., {\bf 162} (1994), 143-187.
\end{thebibliography}
\end{document}